\documentclass[11pt]{article}%
\usepackage{amssymb,amsmath,amsfonts,amsthm,array,bm,bbm,color}%
\usepackage{graphicx}
\providecommand{\U}[1]{\protect\rule{.1in}{.1in}}

\numberwithin{equation}{section}

\newtheorem{theorem}{Theorem}[section]
\newtheorem{lemma}[theorem]{Lemma}
\newtheorem{corollary}[theorem]{Corollary}

\newtheorem{remark}[theorem]{Remark}
\newtheorem{example}[theorem]{Example}
\newtheorem{definition}[theorem]{Definition}
\newtheorem{hypothesis}[theorem]{Hypothesis}

\def\<{\langle}
\def\>{\rangle}
\def\E{\mathbb{E}}
\def\P{\mathbb{P}}
\def\R{\mathbb{R}}
\def\T{\mathbb{T}}
\def\Z{\mathbb{Z}}

\begin{document}
	\title{Scaling Limit and Large Deviation for 3D Globally Modified Stochastic Navier-Stokes Equations with Transport Noise}
	\author{Chang Liu\footnote{Email: liuchang2021@amss.ac.cn. School of Mathematical Sciences, University of Chinese Academy of Sciences, Beijing 100049, China and Academy of Mathematics and Systems Science, Chinese Academy of Sciences, Beijing 100190, China}
		\quad Dejun Luo\footnote{Email: luodj@amss.ac.cn. Key Laboratory of RCSDS, Academy of Mathematics and Systems Science, Chinese Academy of Sciences, Beijing 100190, China and School of Mathematical Sciences, University of Chinese Academy of Sciences, Beijing 100049, China}}
	\maketitle
	
	\vspace{-20pt}
	
	\begin{abstract}
		We consider the globally modified stochastic (hyperviscous)  Navier-Stokes equations with transport noise on 3D torus. We first establish the existence and pathwise uniqueness of the weak solutions, and then show their convergence to the solutions of the deterministic 3D globally modified (hyperviscous) Navier-Stokes equations in an appropriate scaling limit. Furthermore, we prove a large deviation principle for the  stochastic globally modified hyperviscous system.
	\end{abstract}
	
	
	\section{Introduction}\label{sec-introduction}
	Let $\T^3:=\R^3/\Z^3$ be the three-dimensional (3D) torus and we consider on $\T^3$ the following stochastic globally modified (hyperviscous) Navier-Stokes equations (stochastic GMNSE):
	\begin{equation}\label{generalized GMNSE}
		\left\{
		\begin{aligned}
			&du+F_N(\|u\|_{H^{1-\delta}}) u\cdot \nabla u \, dt+d \, \nabla p= -(-\Delta)^\Lambda u\, dt+ \circ d W \cdot \nabla u ,  \\
			& \text{div} \, u=0, \quad u(0)=u_0,
		\end{aligned}
		\right.
	\end{equation}
	where $u$ is the velocity field and $p$ is the pressure, $\Lambda \in [1,2)$; $H^s\, (s\in \R )$ is the usual Soboelv space of divergence-free vector fields on $\T^3$, and $H^0$ will be written as $L^2$. For any fixed $N>0$,  the cut-off function $F_N: (0,\infty] \rightarrow [0,1]$ is defined as $F_N(r):=\min \{1,\frac{N}{r}\}$, and we will take $\delta \in (0,\frac{1}{4})$ if $\Lambda=1$ and $\delta \in [0,\frac{1}{4})$ if $\Lambda \in (1,2)$. Besides, $\circ d$ means the Stratonovich stochastic differential, and the noise $W=W(t,x)$ is Brownian in time and colored in space, with the expression
	\begin{equation}\label{noise-W}
		W= \sqrt{\frac{3\nu}{2}}\sum_{k\in \Z^3_0} \sum_{i=1}^{2} \theta_k \sigma_{k,i}(x) W_t^{k,i},
	\end{equation}
	where $\nu>0$ is a fixed constant, $\theta= \{\theta_k\}_k\in \ell^2(\Z^3_0)$ is a square summable sequence indexed by $\Z^3_0= \Z^3\setminus \{0\}$, $\{\sigma_{k,i}: k \in \Z_0^3, i=1,2\}$ are complex divergence-free vector fields whose precise expressions will be given below, $\{W^{k,i}: k \in \Z_0^3, i=1,2\}$ are independent complex-valued Brownian motions defined on some filtered probability space $(\Omega,\mathcal{F},(\mathcal{F}_t)_t, \P)$. In this paper, we will establish the well-posedness of the equation \eqref{generalized GMNSE} and prove that, under an appropriate scaling limit for the noise, its solution converges weakly to that of a deterministic equation with enhanced dissipation. Additionally, we will present a large deviation principle (LDP) for the hyperviscous stochastic GMNSE, i.e., $\Lambda>1$ in \eqref{generalized GMNSE}.
	
	The deterministic Navier-Stokes equations have been widely studied due to their fundamental role in fluid dynamics, with applications in areas such as weather modeling and aircraft design. In two dimensions, the existence and uniqueness of Leray weak solutions have been well established, see e.g., \cite{KukShi12, Temam}. However, in the 3D case, the uniqueness of Leray weak solutions is unknown due to the challenges posed by nonlinear term; therefore, several different forms of regularized (stochastic) fluid models were introduced, see e.g. \cite{CHOT05, GMNSE06, OlsTiti07, RocZha09}. In \cite{GMNSE06}, Caraballo, Real and Kloeden introduced the following deterministic 3D GMNSE with an external force field $f$:
	\begin{equation}\label{deterministic GMNSE}
		\left\{
		\begin{aligned}
			&\partial_t u +F_N(\|u\|_{H^{1}}) u\cdot \nabla u+\nabla p= \Delta u+ f ,  \\
			& \text{div} \, u=0, \quad u(0)=u_0.
		\end{aligned}
		\right.
	\end{equation}
	Thanks to the cut-off in nonlinearity, they were able to prove the global existence and uniqueness of strong solutions for $u_0 \in V:= H^1$ and $f \in L^2([0,T];L^2)$. Subsequently, they studied in \cite{CRK08} the invariant measures and statistical solutions for the system, attempting to establish connections with global attractors and time-average measures. The authors of  \cite{KLR07} explored the asymptotic behavior of solutions to the 3D GMNSE and proved the existence of pullback $V$-attractors. In \cite{MR09}, Romito exploited the nice properties of the nonlinear term due to the cut-off and demonstrated the uniqueness of weak solutions for equation \eqref{deterministic GMNSE} with $u_0 \in L^2$ and $f\in L^2([0,T];L^2)$.
	
	Stochastic fluid equations with transport-type noise have been studied in the early papers \cite{BCF92, MR04, MR05}, see \cite{BFM16, FGP10, LangCri23} for more recent developments and also \cite{DFV14, FGP11} for other models. We refer to \cite{FP21,FP22,DDHolm 15} for some rigorous derivations of fluid dynamics models with transport-type noise in Stratonovich form. In these equations, transport noise models the effect of turbulent small-scale fluid structures on the large-scale fluid components. In \cite{Gal20}, Galeati considered linear transport equations with Stratonovich transport noise and proved that, under suitable rescaling of the space covariance of the noise, the solutions converge weakly to the unique solution of a deterministic parabolic equation. An interesting feature of this result is that the noise part gives rise to an extra dissipation term. Since then, this method has been applied in many works to study the regularizing effects of transport noise on various models, see e.g. \cite{CL23, FGL21 JEE, FLL24, Luo21, PFH23, QS24} and \cite{FlanLuongo23} for lecture notes on this topic. These results partially confirm Boussinesq's hypothesis \cite{Boussinesq} that turbulent small-scale fluctuations are dissipative on the mean part of fluid. Using the mild formulations of approximate stochastic equations and the deterministic limit equations, one can also establish quantitative convergence estimates \cite{FLD quantitative, LT23 NA, ZH24}. Furthermore, the strong dissipation term emerging from scaling limit of noise has been used to suppress possible blow-up of some nonlinear equations, cf. \cite{Agresti 24, Agr24b, FHLN22, FL21 PTRF, Lange 24, XieGao}. Finally, we remark that in the vorticity formulation of 3D fluid models, physically relevant noise should involve an additional stretching part, for which the rigorous treatment is much more delicate; see \cite{BFL24, BFLT24, ButLuon24, FlanLuo24} for some recent progress.
	
	Inspired by these works, we consider the 3D (hyperviscous) stochastic GMNSE \eqref{generalized GMNSE}. We are mainly concerned with the case $\Lambda=1$ where the cut-off function contains a weaker norm of solutions compared to \eqref{deterministic GMNSE}, but our proof works for general $\Lambda\in [1,2)$. We will first show the existence of solutions to \eqref{generalized GMNSE} which are weak in both analytic and probabilistic sense; we can also prove the pathwise uniqueness and thus obtain probabilistically strong solutions. Then, we aim to establish a scaling limit result under some assumptions on the noise coefficients; since the solutions are defined on the same probability space $(\Omega,\mathcal{F}, \P)$ as the noise $W$, we deduce that the unique solutions to \eqref{generalized GMNSE} converge to that of the deterministic equation, strongly in $L^p(\Omega,\mathcal X)$ for some suitable functional space $\mathcal X$, see Theorem \ref{scaling limit}.
	
	Furthermore, we want to study LDP of \eqref{generalized GMNSE} in the case $\Lambda \in (1,2)$, and the cut-off norm will be taken as $\|u \|_{H^1}$ for simplicity; see Remark \ref{rmk on LDP GMNSE} below for the reason of restricting to such $\Lambda$. We refer to \cite{Varadhan 1984} for a comprehensive introduction to LDP theory and its applications. The weak convergence method, initially developed in \cite{BD2000, BDM08}, has been widely applied in subsequent works to establish LDP for the laws of solutions to various SPDEs, as demonstrated in \cite{CM10, CD19, GL24 LDP}. Our purpose is to apply this method to establish an LDP for the 3D hyperviscous stochastic GMNSE.

	\subsection{Main results}\label{subsec-main results}
	We first introduce some frequently used notations. Recall that $\ell^2=\ell^2(\Z^3_0)$ is the space of square summable real sequences, $H^s(s\in \R)$ is the usual Sobolev space of divergence-free vector fields on $\T^3$, which will be endowed with the norm $\|\cdot \|_{H^s}$ and we often write $\|\cdot\|_{H^0}$ as $\|\cdot\|_{L^2}$. The notation $\< \cdot, \cdot \>$ stands for the inner product in $L^2$ or the duality between elements in $H^s$ and $H^{-s}$. As the equations preserve the spatial average of solutions, we shall assume that these Sobolev spaces consist of vector fields with zero spatial average. Next, the notation $a \lesssim b$ implies that there exists a constant $C>0$, such that $a \leq Cb$; if we want to emphasize the dependence of $C$ on some parameter $\gamma$, then we write $a\lesssim_{\gamma} b$. Finally, to simplify the expression, we will use $\sum_{k,i}$ for $\sum_{k\in \Z^3_0} \sum_{i=1}^{2}$ in the sequel.

	Our first result is the well-posedness of equation \eqref{generalized GMNSE}. To simplify the analysis, we begin by rewriting it in an equivalent but more convenient form. Substituting the noise \eqref{noise-W} into \eqref{generalized GMNSE} and applying Leray's projection operator $\Pi$ to both sides of the equation, we obtain
	$$du+F_N(\|u\|_{H^{1-\delta}}) \Pi(u\cdot \nabla u) \, dt=-(-\Delta)^\Lambda u\, dt+\sqrt{\frac{3\nu}{2}}\sum_{k,i} \theta_k \Pi(\sigma_{k,i} \cdot \nabla u) \circ dW^{k,i}_t $$
	since $u$ is divergence-free. Furthermore, we can write the above equation in It\^o form:
	\begin{equation}\label{Ito form}
		du+F_N(\|u\|_{H^{1-\delta}}) \Pi(u\cdot \nabla u) \, dt= -(-\Delta)^\Lambda u\, dt+\sqrt{\frac{3\nu}{2}}\sum_{k,i}  \theta_k \Pi(\sigma_{k,i} \cdot \nabla u) \, dW_t^{k,i}+S_\theta(u) \, dt,
	\end{equation}
	where $S_\theta(u)$ is the Stratonovich-It\^o corrector and has the following expression:
	\begin{equation}\label{Stheta def}
		S_\theta(u)=\frac{3\nu}{2}\sum_{k,i} \theta_k^2\, \Pi \big[\sigma_{k,i}\cdot \nabla \Pi(\sigma_{-k,i} \cdot \nabla u)\big].
	\end{equation}
	If there is no Leray projection in the identity above, one can verify quite easily that for the $\{\sigma_{k,i}\}_{k,i}$ and $\{\theta_k\}_k$ specified in Section \ref{subsec-sturcture of noise}, $S_\theta(u)= \nu \Delta u$. Although it seems more complicated now, Flandoli and Luo \cite{FL21 PTRF} proved that for smooth vector field $\phi$ and a specific sequence of $\{\theta^n\}_{n\geq 1}\subset \ell^2(\Z^3_0)$, $S_{\theta^n}(\phi)$ converges in $L^2$ to $\frac35 \nu \Delta\phi$ as $n\rightarrow \infty$. Later, the authors of  \cite{Luo23 JDDE} and \cite{LT23 NA} established a quantitative convergence rate for two different choices of $\{\theta^n\}_{n\geq 1}$,  see Theorem \ref{S limit thm} below for instance. Additionally, this operator is symmetric with respect to the $L^2$-inner product of divergence-free vector fields.

	Before presenting the well-posedness result, we define the solution to \eqref{Ito form}, which is strong in the probabilistic sense but weak in the analytic sense.
	
	\begin{definition}\label{def of solutions}
		Let $\Lambda$ and $\delta$ be as mentioned below \eqref{generalized GMNSE}. We say that equation \eqref{Ito form} has a weak solution, if for a given filtered probability space $(\Omega,\mathcal{F},(\mathcal{F}_t),\P)$ and a sequence of independent $(\mathcal{F}_t)$-complex Brownian motions $\{W^{k,i}: k \in \Z_0^3, i=1,2\}$, we can find an $(\mathcal{F}_t)$-progressively measurable process $u$ with trajectories in $L^\infty([0,T],L^2) \cap L^2([0,T];H^\Lambda)$, such that for any divergence-free vector field $\phi\in C^\infty$, the following equality holds $\P$-a.s.:
		$$\aligned
		\<u(t),\phi\>&=\< u_0, \phi\>+ \int_{0}^{t} F_N(\|u(s)\|_{H^{1-\delta}})\<u(s) \cdot \nabla  \phi, u(s)\> \, ds- \int_{0}^{t} \<u(s),(-\Delta)^\Lambda  \phi \> \, ds\\
		&\quad +\int_{0}^{t} \<u(s),S_\theta(\phi) \> \, ds-\sqrt{\frac{3\nu}{2}}\sum_{k,i} \theta_k \int_{0}^{t} \<u(s),\sigma_{k,i} \cdot \nabla \phi\> \, dW^{k,i}_s.
		\endaligned $$
	\end{definition}
	
	Then we can state the well-posedness result for (hyperviscous) stochastic GMNSE  \eqref{Ito form}.
	\begin{theorem}\label{well-posedness thm}
		Given any $u_0 \in L^2$ and $T>0$, for $\Lambda \in [1,2)$, there exists a pathwise unique weak solution $u$ to equation \eqref{Ito form} in the sense of Definition \ref{def of solutions}.
	\end{theorem}
	
	Compared to the well-posedness results for the 3D deterministic GMNSE in \cite{GMNSE06}, we consider the (hyperviscous) equation with transport noise. We first prove the existence of probabilistically weak solution and then establish the pathwise uniqueness, which implies the existence of probabilistically strong solution. The proof of weak existence makes use of the classical Galerkin approximations and the compactness approach. For the proof of pathwise uniqueness, we follow the idea of Romito \cite{MR09} and show that the usual energy estimate and Gr\"onwall lemma work also for the cut-off with weaker norm. Further details can be found in Section \ref{sec-wellposedness}.
	
	To prove our second result on scaling limit, we consider a sequence of noise coefficients $\{\theta^n\}_{n\geq 1}$ satisfying $\|\theta^n\|_{\ell^2}=1$ and
	\begin{equation}\label{assumption on theta}
		\lim_{n\rightarrow\infty} \|\theta^n\|_{\ell^\infty}=0.
	\end{equation}
	Specifically, as in \cite{FL21 PTRF, LT23 NA}, we choose the following explicit form in this paper:
	\begin{equation}\label{thetaR expression}
		\theta_k^n=\sqrt{\varepsilon_n} \frac{1}{|k|^r} \textbf{1}_{\{1\leq |k|\leq n\}},\quad k\in \Z^3_0,
	\end{equation}
	where $r \in (0,\frac{3}{2})$ and $\varepsilon_n$ is a normalizing constant such that $\|\theta^n\|_{\ell^2}=1$. It can be verified that this choice satisfies the assumption \eqref{assumption on theta} (see Example \ref{example for theta} below). Then for $\Lambda$ and $\delta$ defined below \eqref{generalized GMNSE}, we consider the sequence of equations with the same initial values $u^n(0)=u_0\in L^2$:
	\begin{equation}\label{sequence of thetaR}
		\begin{split}
			du^n+F_N(\|u^n\|_{H^{1-\delta}}) \Pi\big(u^n\cdot \nabla u^n\big) \, dt=& -(-\Delta)^\Lambda u^n \, dt+S_{\theta^n}(u^n) \, dt\\
			&+\sqrt{\frac{3\nu}{2}}\sum_{k,i}  \theta_k^n \Pi(\sigma_{k,i} \cdot \nabla u^n) \, dW_t^{k,i}.
		\end{split}
	\end{equation}
	Due to the choice \eqref{thetaR expression} of coefficients, the sum over $k$ is in fact restricted to $1\le |k|\le n$. Theorem \ref{well-posedness thm} implies that the above equation admits a pathwise unique and probabilistically strong solution $\{u^n(t) \}_{t\in [0,T]}$.
	
	For $\gamma \in (0,\frac{1}{2})$, we define the space
	$$\mathcal{X}:=L^2([0,T];H^{1-\delta}) \cap C([0,T],H^{-\gamma}),$$
	equipped with the norm $\| \cdot \|_{\mathcal{X}}=\|\cdot\|_{L^2([0,T];H^{1-\delta})} \vee \|\cdot\|_{C([0,T];H^{-\gamma})}$. Here $\delta \in (0,\frac{1}{4})$ for stochastic GMNSE (i.e., \eqref{generalized GMNSE} with $\Lambda=1$) and $\delta \in [0,\frac{1}{4})$ for stochastic hyperviscous GMNSE (i.e., \eqref{generalized GMNSE} with $\Lambda \in (1,2)$).
	
	\begin{theorem}\label{scaling limit}
		For any $\Lambda \in [1,2)$ and $T>0$, the solutions $\{u^n\}_{n\geq 1}$ of equations \eqref{sequence of thetaR}  converge to the unique weak solution $\bar{u} \in L^\infty([0,T];L^2)\cap L^2([0,T];H^\Lambda)$ of the following deterministic (hyperviscous) GMNSE as $n\rightarrow \infty$:
		\begin{equation}\label{limit eq}
			\partial_t\bar{u}+F_N(\|\bar{u}\|_{H^{1-\delta}}) \Pi(\bar{u}\cdot \nabla \bar{u})=-(-\Delta)^\Lambda \bar{u}+\frac{3\nu }{5}\Delta \bar{u}, \quad \bar{u}(0)=u_0.
		\end{equation}
		More precisely, for any $p<+\infty$, the following limit holds:
		$$	\lim_{n\to \infty} \E \Big[\|u^n-\bar{u}\|_{\mathcal{X}}^p\Big]= 0.$$
	\end{theorem}
	
	\begin{remark}\label{rmk on scaling limit}
		We consider for simplicity the sequence of stochastic equations \eqref{sequence of thetaR} and limit equation \eqref{limit eq} with the same initial values $u^n(0)=\bar{u}(0)=u_0$. In fact, the scaling limit result remains valid as long as $u^n(0) \rightharpoonup \bar{u}(0)$ in $L^2$.
	\end{remark}
	
	Theorem \ref{scaling limit} shows that the solutions $u^n$ converge strongly in $L^p(\Omega, \mathcal X)$ to the limit $u$ as $n\rightarrow \infty$. However, this is only a qualitative result without offering a quantitative convergence rate.  For the case $\Lambda=1$ and $\delta\in (0,\frac{1}{4})$, by expressing the solutions to equations \eqref{sequence of thetaR} and \eqref{limit eq} in mild form and then taking the difference between them, we can utilize standard heat kernel estimates to arrive at
	$$\E \, \|u^n-\bar{u}\|^2_{L^2([0,T];H^{1-\delta})} \lesssim 	\nu^2 n^{-2\delta} N^{-2} +\nu^{\frac{\delta}{2}+2} \delta^{\frac{\delta-5}{5}} T^{\frac{7\delta-\delta^2}{10}} N^{-2} \|\theta^n\|_{\ell^\infty}^{\frac{2\delta}{5}},$$
	which, as $n\rightarrow \infty$, can be sufficiently small under appropriate choices of the coefficients (e.g.,  $\nu \gg 1$ being fixed). The details and the corresponding proofs can be found in Section \ref{subsec-quantitative conv}. We remark that, due to the strong cut-off norm in nonlinearity, the approach in \cite{FLD quantitative} cannot be applied here to derive an estimate on $\|u^n-\bar{u}\|^2_{C([0,T],H^{-\gamma})}$ for some $\gamma>0$.
	
	Finally, we will provide an LDP result for \eqref{generalized GMNSE}. As mentioned above, we consider the case $\Lambda \in (1,2)$ and $\delta=0$.  The reason why our method cannot be applied to the stochastic GMNSE (i.e., \eqref{generalized GMNSE} with $\Lambda=1$) will be explained in Remark \ref{rmk on LDP GMNSE} below.
	
	In view of the scaling limit result in Theorem \ref{scaling limit}, for $g \in L^2([0,T];H^r)$ and $\Lambda \in (1,2)$, the corresponding skeleton equation reads as follows:
	\begin{equation}\label{skeleton eq}
		\left\{
		\begin{aligned}
			&\partial_t u+ F_N(\|u\|_{H^{1}}) \Pi (u \cdot \nabla u)=-(-\Delta)^\Lambda u+\frac{3\nu}{5}\Delta u+\Pi(g\cdot\nabla u), \\
			& \text{div} \, u=0, \quad u(0)=u_0.
		\end{aligned}
		\right.
	\end{equation}
	
	\begin{remark}\label{rmk on LDP GMNSE}
		To apply the weak convergence method for proving LDP results, we first need to establish the well-posedness of the skeleton equation. However, for stochastic GMNSE \eqref{generalized GMNSE} with $\Lambda=1$ and $\delta\in (0, \frac14)$, we are unable to prove the uniqueness of the corresponding skeleton equation. Specifically, suppose $u_1$ and $u_2$ are both solutions to equation \eqref{skeleton eq} with $\Lambda = 1$, and let $\xi:= u_1 - u_2$. Before applying the energy equality, it is necessary to verify that the following condition from \cite[Theorem 2.12]{RL 1990} holds:
		$$\int_{0}^{T} \|\xi (t)\|_{H^1}^2+\|\partial_t \xi\|_{H^{-1}}^2 \, dt <+\infty.$$
		But for $r\in (0,\frac{3}{2})$, the second term on the left-hand side is not finite for $g\in L^2([0,T];H^r)$ and $\xi \in L^2([0,T];H^1)$ due to the following estimate:
		$$\int_{0}^{T} \big\|\Pi(g\cdot \nabla \xi) \big\|_{H^{-1}}^2 \, dt \lesssim \int_{0}^{T} \|g\otimes\xi \|_{L^2}^2\, dt \lesssim \int_{0}^{T} \|g\|_{H^r}^2 \|\xi\|_{H^{\frac{3}{2}-r}}^2 \, dt.$$
		So we cannot apply the energy equality to establish the uniqueness of the skeleton equation.
	\end{remark}
	
	Based on the above arguments, we will only discuss the hyperviscous stochastic GMNSE in our final result concerning LDP. To state the theorem, we introduce the following two spaces:
	\begin{equation}\label{ETL space}
		\mathcal{E}_0^L=\big\{ u_0\in L^2: \|u_0\|_{L^2} \leq L\big\}, \quad \mathcal{E}^{T,L}=\Big\{ u\in C_w([0,T];L^2):\sup_{t\in [0,T]} \|u(t) \|_{L^2}\leq L\Big\},
	\end{equation}
	where $w$ stands for weak continuity in time. These two spaces are equipped with the topologies of  $H^{-\lambda}$ and $C([0,T];H^{-\lambda})$ for some $\lambda>0$, respectively.
	Recall that the noise coefficients $\{\theta^n\}_n$ are defined  in \eqref{thetaR expression}, and $\varepsilon_n$ is the parameter mentioned therein. Define space-time noise
	$$W^r(t,x)=\sqrt{\frac{3\nu}{2}}\sum_{k,i} |k|^{-r} \sigma_{k,i} (x) W^{k,i}_t.$$
	Then the equation \eqref{sequence of thetaR} can be written more compactly as
	\begin{equation}\label{eq-LDP}
		\begin{split}
			du^n+F_N(\|u^n\|_{H^{1-\delta}}) \Pi\big(u^n\cdot \nabla u^n\big) \, dt=& -(-\Delta)^\Lambda u^n \, dt+S_{\theta^n}(u^n) \, dt + \sqrt{\varepsilon_n}\, d\Pi(W^r_n \cdot \nabla u^n),
		\end{split}
	\end{equation}
	where $W^r_n= \Pi_n W^r$ and $\Pi_n$ is the projection to $H_n:=\mbox{span}\{\sigma_{k,i}: 1\le |k|\le n, i=1,2\}$.
	
	Denote $\mathcal{G}^0$ as the solution map associated to skeleton equation \eqref{skeleton eq} with $u_0 \in L^2$ and $g\in L^2([0,T];H^r)$; and $\mathcal{G}^n$ stands for the solution map associated to equation  \eqref{eq-LDP} with parameters $\Lambda \in (1,2)$, $\delta=0$ and initial value $u_0^n \in L^2$. All these solution maps are measurable and we will verify in Section \ref{subsec-prf of LDP} that they make sense. Finally, we denote $\text{Int}(g)(\cdot)=\int_{0}^{\cdot} g(s) \, ds$.
	Now we can state the result on LDP.
	\begin{theorem}\label{LDP thm}
		For $\Lambda \in (1,2)$ and $\delta=0$, the solutions $\{u^n=\mathcal{G}^n(u_0^n, \sqrt{\varepsilon_n}\, W^r)\}_{n\geq 1}$ of equation \eqref{eq-LDP} with initial value $\{u_0^n\}_{n\geq 1} \subset \mathcal{E}_0^L$ satisfy the uniform large deviation principle with the rate function
		$$	I_{u_0} (u)=\inf_{\{g\in L^2([0,T];H^r):\, u=\mathcal{G}^0(u_0, \text{Int}(g))\}} \bigg\{ \frac{1}{2} \int_{0}^{T} \|g(s) \|_{H^r}^2 \, ds \bigg\},$$
		where $u \in \mathcal{E}^{T,L}$ and $r \in (0,\frac{3}{2})$.
	\end{theorem}
	We will prove this theorem by following \cite[Section 2]{GL24 LDP}, where the authors study LDP for stochastic transport equations and 2D Euler equations, both driven by transport noise. Compared to their work, we consider the nonlinear equation in a higher-dimensional setting. Moreover, for the equations considered in \cite{GL24 LDP}, the pathwise uniqueness of  solutions is unknown for $L^2$-initial value, hence they first address equations with $L^\infty$-initial value and prove the LDP result by using weak convergence approach, and then extend the conclusion to $L^2$-initial data via $\Gamma$-convergence method. Here, based on the well-posedness result for the hyperviscous stochastic GMNSE with $L^2$-initial value, we can directly apply the weak convergence method to obtain an LDP result.

	\subsection{Organization of our paper}\label{subsec-organization}
	In Section \ref{subsec-sturcture of noise}, we introduce the structure of the noise along with an example, and present a key theorem that will be used in the subsequent proofs; Section \ref{subsec-analytical results} outlines several classical analytical results.  Section \ref{sec-wellposedness} focuses on establishing the well-posedness of equation \eqref{generalized GMNSE},  with the existence and pathwise uniqueness discussed in Sections \ref{subsec-exist} and \ref{subsec-unique}, respectively. In Section \ref{sec-convergence}, we first prove the scaling limit for stochastic GMNSE in Section \ref{subsec-scaling limit}, followed by the estimate of quantitative convergence rate in Section \ref{subsec-quantitative conv}. Finally, we will introduce the weak convergence method in Section \ref{subsec-weak convergence method}, and the proof of LDP for the hyperviscous stochastic GMNSE will be given in Section \ref{subsec-prf of LDP}.

	\section{Preliminaries}\label{sec-preliminaries}
	\subsection{Structure of the noise}\label{subsec-sturcture of noise}
	Recall that $\Z_0^3=\Z^3 \setminus \{0\}$ is the nonzero lattice points and $\theta \in \ell^2(\Z_0^3)$. In the sequel, we always assume that $\theta$ is radially symmetric with only finitely many nonzero components, and $\|\theta\|_{\ell^2}=1$.
	
	Suppose that $\Z_0^3=\Z_+^3 \cup \Z_-^3$ is a disjoint partition of $\Z_0^3$ satisfying $\Z_+^3 =-\Z_-^3$. Denote $e_k(x)=e^{2\pi i k\cdot x}, \, x\in \T^3, \, k\in \Z_0^3$. For any $k \in \Z_+^3$, let $\{a_{k,1},a_{k,2}\}$ be an orthonormal basis of $k^\perp:=\{x\in \R^3: k \cdot x=0\}$. For $k \in \Z_-^3$, we define $a_{k,i}=a_{-k,i}$ for $i=1,2$. Then the divergence-free vector fields take the form of
	$$\sigma_{k,i}=a_{k,i} e_k(x), \quad x\in \T^3, \ k\in \Z_0^3, \ i=1,2,$$
	and the family $\{\sigma_{k,i}: k \in \Z_0^3,i=1,2\}$ is  the complete orthonormal system (CONS) of the space $L^2(\T^3,\mathbb{C}^3)$ with zero mean.
	
	For the complex valued Brownian motions, they can be defined as
	$$W^{k,i}= \left \{
	\aligned
	&B^{k,i}+ \sqrt{-1} B^{-k,i}, \quad k \in \Z_+^3,\\
	&B^{-k,i}-\sqrt{-1} B^{k,i}, \quad k \in \Z_-^3,
	\endaligned
	\right.$$
	where $\{B^{k,i}: k\in \Z_0^3, i=1,2\}$ is a family of independent standard real Brownian motions. Besides, the following identity holds:
	$$\big[W^{k,i},W^{l,j}\big]_t=2t \delta_{k,-l} \delta_{i,j}, \quad k,l \in \Z_0^3, \ i,j \in \{1,2\}.$$
	
	In the next example, we show how to estimate the $\ell^\infty$-norm of $\theta^n$ defined in \eqref{thetaR expression}.
	
	\begin{example}\label{example for theta}
		Let $\theta_k^n$ be given in \eqref{thetaR expression}, then by the assumption that $\|\theta^n\|_{\ell^2}=1$, we can explicitly calculate
		$$\varepsilon_n=\bigg(\sum_{1\leq |k| \leq n} \frac{1}{|k|^{2r}}\bigg)^{-1} \sim \bigg(\int_{1}^{n} \rho^{2-2r} \, d\rho \bigg)^{-1} \sim n^{2r-3},$$
		which implies that $\|\theta^n\|_{\ell^\infty}=\sqrt{\varepsilon_n} \rightarrow 0$ as $n \rightarrow \infty$.
	\end{example}
	
	Recall the definition of the It\^o-Stratonovich corrector $S_\theta$ in \eqref{Stheta def}. The following result and its proof can be found in \cite[Theorem 3.5]{LT23 NA}.
	\begin{theorem}\label{S limit thm}
		For any $n \geq 1$, let $\theta^n$ be defined as in \eqref{thetaR expression}. Then for any $\alpha \in [0,1]$ and $b \in \R$, the following estimate holds for any divergence free field $\phi \in H^b(\T^3,\R^3)$:
		$$ \Big\| S_{\theta^n}(\phi)-\frac{3}{5}\nu \Delta \phi \Big\|_{H^{b-2-\alpha}} \lesssim \nu n^{-\alpha} \|\phi\|_{H^b}.$$
	\end{theorem}

	\subsection{Analytical results}\label{subsec-analytical results}
	In this section, we collect some fundamental results that will be used beow. To begin with, we present several classical inequality, whose proofs can be found in standard harmonic analysis textbooks; hence we omit them here.
	\begin{lemma}\label{HHH}
		For any $s_1,s_2 <\frac{d}{2}$, if $s_1+s_2>0$, then for any $f \in H^{s_1}(\T^d)$ and $h \in H^{s_2}(\T^d)$, we have $fh \in H^{s_1+s_2-\frac{d}{2}}(\T^d)$, and the following inequality holds:
		$$	\|fh\|_{H^{s_1+s_2-\frac{d}{2}}(\T^d)} \lesssim \|f\|_{H^{s_1}(\T^d)} \|h\|_{H^{s_2}(\T^d)}. $$
	\end{lemma}
	
	\begin{lemma}[Interpolation inequality] \label{interpolation}
		For any $s_1<s<s_2$, there exists $\tau \in (0,1)$ satisfying $s=\tau s_1+(1-\tau)s_2$, such that
		$$	\|f\|_{H^{s}} \leq \|f\|^{\tau}_{H^{s_1}} \|f\|^{1-\tau}_{H^{s_2}}.	$$
	\end{lemma}
	
	Next, we provide several compact embedding results from \cite[Corollary 5]{JSimon} without proofs. Recall that
	fractional Sobolev spaces are defined as
	$$W^{s,p} ([0,T];Y)=\Big\{ f\in L^p([0,T];Y): \int_{0}^{T}  \int_{0}^{T} \frac{\|f_t-f_s\|_Y^p}{|t-s|^{1+sp}} \, dtds<\infty \Big\}, $$
	where $Y$ is a Banach space and $s\in (0,1)$, $p \in [1,+\infty)$, endowed with the norm
	$$\|f\|_{W^{s,p} ([0,T];Y)}^p=\|f\|_{L^p([0,T];Y)}^p+\int_{0}^{T} \int_{0}^{T} \frac{\|f_t-f_s\|_Y^p}{|t-s|^{1+sp}} \, dtds.$$
	Then we have
	
	\begin{theorem}\label{compact embedding}
		(i) Suppose $\alpha \in (0,1/2)$, $\beta>5/2$ and $\delta\in (0,1/4)$, then the following embedding is compact:
		$$L^2([0,T];H^1) \cap W^{\alpha,2} ([0,T];H^{-\beta}) \subset L^2([0,T];H^{1-\delta}).$$
		
		(i') Suppose that $\alpha, \beta$ are the same as in (i), then for $\Lambda \in (1,2)$, the following embedding is compact:
		$$L^2([0,T];H^\Lambda) \cap W^{\alpha,2} ([0,T];H^{-\beta}) \subset L^2([0,T];H^{1}).$$
		
		(ii) Suppose $\gamma \in (0,1/2)$, $\beta>5/2$ and $p>12(\beta-\gamma)/\gamma$,  then the following embedding is compact:
		$$L^p([0,T];L^2) \cap W^{\frac{1}{3},4}([0,T];H^{-\beta}) \subset C([0,T];H^{-\gamma}).$$
	\end{theorem}
	
	Finally, we introduce two properties of the heat semigroup $\{e^{t \Delta}\}_{t\geq 0}$.
	
	\begin{lemma}\label{semigroup property}
		For any $s_1<s_2$, $\tau \in \R$, $\lambda>0$ and $f \in L^2([s_1,s_2];H^\tau)$, we have
		$$\int_{s_1}^{s_2} \bigg\|\int_{s_1}^{t} e^{\lambda (t-z)\Delta} f_z \, dz \bigg\|^2_{H^{\tau+2}}  dt \lesssim \lambda^{-2} \int_{s_1}^{s_2} \|f_z\|_{H^\tau}^2 \, dz.$$
	\end{lemma}
	
	\begin{lemma}\label{semigroup prop 2}
		Let $f\in H^\tau$, $\tau \in \R$. For any $\rho \geq 0$, it holds
		$$\big\|e^{t\Delta} f\big\|_{H^{\tau+\rho}} \lesssim t^{-\rho/2} \|f\|_{H^\tau}.$$
	\end{lemma}

	\section{Well-posedness}\label{sec-wellposedness}
	We prove Theorem \ref{well-posedness thm} in this section. Noting that the dissipation will be enhanced as $\Lambda$ increases, to show the well-posedness, it is enough for us to consider equation \eqref{Ito form} with $\Lambda=1$ and $\delta \in (0,\frac{1}{4})$. Since the proof is quite long, we will show the existence and pathwise uniqueness in Sections \ref{subsec-exist} and \ref{subsec-unique}, respectively.

	\subsection{Existence} \label{subsec-exist}
	To prove the existence of weak solutions, we use classical Galerkin approximations and compactness method. Recalling the space $H_m$ and projection $\Pi_m$ defined below the equation \eqref{eq-LDP}, we consider the Galerkin approximations of equation \eqref{Ito form} with $\Lambda=1$ and initial value $u_m(0)=\Pi_m u_0$:
	\begin{equation}\label{finite dimension}
		du_m+F_N(\|u_m\|_{H^{1-\delta}}) \Pi_m(u_m\cdot \nabla u_m) \, dt= \Delta u_m \, dt+\sqrt{\frac{3\nu}{2}}\sum_{k,i}  \theta_k \Pi_m(\sigma_{k,i} \cdot \nabla u_m) \, dW_t^{k,i}+S_\theta (u_m) \, dt.
	\end{equation}
	Here we mention that if $u_m \in H_m$, then $S_\theta (u_m) \in H_m$ (see \cite[Section 5]{FL21 PTRF}), so we do not have to make projection $\Pi_m$ on $S_\theta(u_m)$ again. Now we give the following lemma for the above equation. Recall that $H^s\ (s\in \R)$ consists of divergence-free vector fields.
	
	\begin{lemma}\label{finite energy estimate}
		Suppose $u_0 \in L^2$, then for every  $T\geq 0$, we have the following uniform estimate:
		$$\{u_m\}_{m\geq 1} \subset L^\infty([0,T];L^2) \cap L^2([0,T];H^1), \quad \P \text{-a.s.}.$$
		To be more precise, it holds
		$$ \sup_{m\geq 1} \sup_{t\in [0,T]}  \bigg[\|u_m(t)\|_{L^2}^2+2 \int_0^t \|\nabla u_m(s)\|_{L^2}^2 \, ds\bigg] \leq \|u_0\|_{L^2}^2, \quad \P \text{-a.s.}.$$
	\end{lemma}
	
	\begin{proof}
		Fixing any $m\geq 1$, we make energy estimate for the equation \eqref{finite dimension}:
		\begin{equation}\label{energy estimate decomposition}
			\begin{split}
				d\|u_m\|_{L^2}^2&=2\<u_m, du_m\>+\<du_m,du_m\>\\
				&=-2\big\<u_m,F_N(\|u_m\|_{H^{1-\delta}}) \Pi_m(u_m\cdot \nabla u_m)\big\>  dt-2 \|\nabla u_m\|_{L^2}^2 \, dt+2\big\<u_m,S_\theta(u_m)\big\> dt\\
				&\quad+2\sqrt{\frac{3\nu}{2}}\sum_{k,i} \theta_k \big\<u_m,\Pi_m(\sigma_{k,i} \cdot \nabla u_m) \big\> dW^{k,i}_t+ 3\nu\sum_{k,i} \theta_k^2 \big\|\Pi_m(\sigma_{k,i} \cdot \nabla u_m) \big\|_{L^2}^2 dt.
			\end{split}
		\end{equation}
		Using the facts that $\text{div} \, \sigma_{k,i}=0$ and $\text{div} \, u_m=0$, the first term and the martingale part on the right-hand side vanish. Besides, noticing the estimates $ \|\Pi_m(\sigma_{k,i} \cdot \nabla u_m) \|_{L^2} \leq  \|\Pi(\sigma_{k,i} \cdot \nabla u_m) \|_{L^2}$ and
		$$\big\<u_m, S_\theta(u_m) \big\> =-\frac{3\nu}{2}\sum_{k,i} \theta_k^2 \big\<\sigma_{k,i} \cdot \nabla u_m, \Pi (\sigma_{-k,i} \cdot \nabla u_m) \big\>=-\frac{3\nu}{2}\sum_{k,i} \theta_k^2 \big\|\Pi (\sigma_{k,i} \cdot \nabla u_m)\big\|_{L^2}^2,$$
		we have
		$$2\big\<u_m,S_\theta(u_m)\big\> dt+3\nu\sum_{k,i} \theta_k^2 \big\|\Pi_m(\sigma_{k,i} \cdot \nabla u_m) \big\|_{L^2} dt\leq 0.$$
		Hence we obtain
		$d\|u_m\|_{L^2}^2 \leq -2 \|\nabla u_m\|_{L^2}^2 \, dt.$
		Integrating the  energy inequality yields
		$$\|u_m(t)\|_{L^2}^2+2\int_0^t \|\nabla u_m(s)\|_{L^2}^2 \, ds \leq \|u_m(0)\|_{L^2}^2 \leq \|u_0\|_{L^2}^2.$$
		The proof is complete.
	\end{proof}
	
	To apply the compact embedding result in Theorem \ref{compact embedding}, we also need the following lemma.
	\begin{lemma}\label{bdd in sobolev space}
		For every $T> 0$ and $\beta>5/2$, there exists a finite constant $C>0$ depending on $\nu$, $T$, $\|u_0\|_{L^2}$, $\|\theta\|_{\ell^\infty}$, but independent of $m$, such that for any $0\leq s<t \leq T$, the solution to \eqref{finite dimension} satisfies
		$$\E \, \big\|u_m(t)-u_m(s)\big\|_{H^{-\beta}}^4 \leq C |t-s|^2.$$
	\end{lemma}
	
	\begin{proof}
		For any $0\leq s<t \leq T$, equation \eqref{finite dimension} yields
		\begin{equation}\label{umt-ums}
			\begin{split}
				&\quad \  \big\|u_m(t)-u_m(s)\big\|_{H^{-\beta}}\\
				&\leq \int_{s}^{t} F_N(\|u_m(z)\|_{H^{1-\delta}}) \big\|\Pi_m \big(u_m(z) \cdot \nabla u_m(z)\big)\big\|_{H^{-\beta}} \, dz+ \int_{s}^{t} \|\Delta u_m(z)\|_{H^{-\beta}} \, dz\\
				&\quad+\sqrt{\frac{3\nu}{2}} \, \bigg\| \sum_{k,i} \int_{s}^{t} \theta_k \Pi_m(\sigma_{k,i} \cdot \nabla u_m(z)) \, dW^{k,i}_z \bigg\|_{H^{-\beta}} + \int_{s}^{t} \big\|S_\theta(u_m(z))\big\|_{H^{-\beta}} \, dz\\
				&=:I_1+I_2+I_3+I_4.
			\end{split}
		\end{equation}
		We will deal with these four terms respectively. For the first term, noticing that for any $z\in [0,T]$, it holds $F_N(\|u_m(z)\|_{H^{1-\delta}}) \leq 1$, so we use H\"older's inequality to get
		$$ I_1^4 \leq |t-s|^3 \int_{s}^{t} \big\|u_m(z) \cdot \nabla u_m(z)\big\|^4_{H^{-\beta}} \, dz. $$
		Taking any vector field $\phi \in C^\infty$, we have
		$$\aligned
		\big|\big\<u_m \cdot \nabla u_m, \phi\big\>\big|=\big|\big\<u_m \cdot \nabla \phi, u_m\big\>\big| \lesssim \|u_m\|_{H^\frac{1}{2}} \|u_m\cdot \nabla \phi\|_{H^{-\frac{1}{2}}}
		&\lesssim
		\|u_m\|_{H^{\frac{1}{2}}}\|u_m\|_{L^2} \|\nabla \phi\|_{H^1}\\
		&\leq \|u_m\|_{H^1}^{\frac{1}{2}} \|u_m\|_{L^2}^{\frac{3}{2}} \|\phi\|_{H^2},
		\endaligned $$
		which implies that $\|u_m \cdot \nabla u_m\|_{H^{-\beta}} \lesssim \|u_m\|_{H^1}^{\frac{1}{2}} \|u_m\|_{L^2}^{\frac{3}{2}}$ for any $\beta>5/2$. Hence by Lemma \ref{finite energy estimate}, we obtain
		$$ \E  I_1^4 \lesssim |t-s|^3 \|u_0\|_{L^2}^6 \int_{s}^{t} \|u_m\|_{H^1}^2 \, dz \lesssim |t-s|^3 \|u_0\|_{L^2}^8.$$
		
		Then we turn to the second term. For $\beta>5/2$, we have
		$$\E I_2^4 \leq  |t-s|^3 \int_{s}^t \|u_m(z)\|_{L^2}^4 \, dz \leq  |t-s|^4 \|u_0\|_{L^2}^4.$$
		As for the next term, Burkholder-Davis-Gundy inequality yields
		$$\E I_3^4 \lesssim\nu^2 \, \E \bigg[ \Big|\sum_{k,i} \theta_k^2 \int_{s}^{t} \big\| \Pi_m(\sigma_{k,i} \cdot \nabla u_m(z))\big\|_{H^{-\beta}}^2  \, dz\Big|^2 \bigg].$$
		Furthermore, noticing that $\sigma_{k,i}$ is divergence-free, we have
		$$\sum_{k,i} \theta_k^2 \big\| \Pi_m(\sigma_{k,i} \cdot \nabla u_m(z))\big\|_{H^{-\beta}}^2 \leq \|\theta\|_{\ell^\infty}^2 \sum_{k,i}  \|\sigma_{k,i} u_m(z)\|_{H^{-\beta+1}}^2 \lesssim \|\theta\|_{\ell^\infty}^2 \sum_k \|e_k u_m(z)\|_{H^{-\beta+1}}^2;$$
		for $\beta>5/2$, it holds
		$$\sum_k \|e_k u_m(z)\|_{H^{-\beta+1}}^2 \leq \sum_k \sum_l \frac{1}{|l|^{2(\beta-1)}} \big|\<u_m, e_{l-k}\>\big|^2 \lesssim \|u_m(z)\|_{L^2}^2 \sum_l \frac{1}{|l|^{2(\beta-1)}} \lesssim \|u_m(z)\|_{L^2}^2,$$
		which can be combined with the above estimate to obtain
		$$\sum_{k,i} \theta_k^2 \big\| \Pi_m(\sigma_{k,i} \cdot \nabla u_m(z))\big\|_{H^{-\beta}}^2  \lesssim \|\theta\|_{\ell^\infty}^2 \|u_m(z)\|_{L^2}^2.$$
		Hence we apply Lemma \ref{finite energy estimate} to arrive at $\E I_3^4 \lesssim \nu^2  |t-s|^2 \|\theta\|_{\ell^\infty}^4 \|u_0\|_{L^2}^4$.
		
		Finally, we use the method of duality to estimate the remaining term. For any $v \in H^1$, by the divergence-free property of $\sigma_{k,i}$, we use Cauchy-Schwarz inequality to obtain
		\begin{equation}\label{eq-S duality}
			\begin{split}
				\big|\big\<S_\theta(u_m),v\big\>\big|&= \bigg|\frac{3\nu}{2}\sum_{k,i}\theta_k^2\big\<\Pi(\sigma_{-k,i}\cdot \nabla u_m),\sigma_{k,i} \cdot \nabla v\big\>\bigg|\\
				&\lesssim \nu \sum_{k,i} \theta_k^2 \|\sigma_{-k,i} \cdot \nabla u_m\|_{L^2} \|\sigma_{k,i} \cdot \nabla v\|_{L^2} \\
				&\leq \nu \bigg(\sum_{k,i} \theta_k^2 \|\sigma_{-k,i} \cdot \nabla u_m\|_{L^2}^2\bigg)^{1/2} \bigg(\sum_{k,i} \theta_k^2 \|\sigma_{k,i} \cdot \nabla v\|_{L^2}^2\bigg)^{1/2} \\
				&\leq \nu \|\theta\|_{\ell^2}^2 \|\nabla u_m\|_{L^2} \|\nabla v\|_{L^2},
			\end{split}
		\end{equation}
		which implies that $\|S_\theta(u_m)\|_{H^{-\beta}} \lesssim \nu \|u_m\|_{H^1} \|\theta\|_{\ell^2}^2$ for $\beta>5/2$. Since we have assumed that $\|\theta\|_{\ell^2}=1$, by Lemma \ref{finite energy estimate}, it holds
		$$\E I_4^4\lesssim \nu^4  \Big(\int_{s}^{t} \|u_m(z)\|_{H^1} \, dz\Big)^4 \leq \nu^4 |t-s|^2 \Big(\int_{s}^{t} \|u_m(z)\|_{H^1}^2 \, dz\Big)^2  \lesssim \nu^4 \|u_0\|_{L^2}^4 |t-s|^2.$$
		Substituting the above estimates for $I_i^4, \, i=1,\ldots,4$ into \eqref{umt-ums}, we complete the proof.
	\end{proof}
	
	Based on Lemma \ref{bdd in sobolev space}, now we can apply Theorem \ref{compact embedding} to further present
	\begin{corollary}\label{tight}
		If we denote $\eta_m$ as the law of $u_m$ for $m\geq 1$, then $\{\eta_m\}_{m\geq 1}$ is tight on $\mathcal{X}$, which is defined above Theorem \ref{scaling limit} with $\delta \in (0,\frac{1}{4})$.
	\end{corollary}
	
	\begin{proof}
		By Lemma \ref{bdd in sobolev space} and the definition of fractional Sobolev norms, for $\alpha\in (0,1/2)$ and $\beta>5/2$, we have
		$$\aligned
		\E \, \|u_m\|^2_{W^{\alpha,2}([0,T];H^{-\beta})}&=\E \int_0^T \int_0^T \frac{\|u_m(t)-u_m(s)\|_{H^{-\beta}}^2}{|t-s|^{1+2\alpha}} \, dtds+\E \, \|u_m\|_{L^2([0,T];H^{-\beta})}^2\\
		&\lesssim \int_0^T \int_0^T \frac{\big(\E \, \|u_m(t)-u_m(s)\|_{H^{-\beta}}^4\big)^{1/2} }{|t-s|^{1+2\alpha}} \, dtds+T\|u_0\|_{L^2}^2\\
		&\lesssim \int_{0}^{T} \int_{0}^{T} \frac{1}{|t-s|^{2\alpha}} dtds+T\|u_0\|_{L^2}^2 <+\infty.
		\endaligned $$
		Similarly, we can verify that
		$$\aligned
		\E \, \|u_m\|^4_{W^{\frac{1}{3},4}([0,T];H^{-\beta})} &=\int_{0}^{T} \int_{0}^{T} \frac{\E \, \|u_m(t)-u_m(s)\|_{H^{-\beta}}^4}{|t-s|^{1+\frac{4}{3}}} \, dtds+\E \, \|u_m\|^4_{L^4([0,T];H^{-\beta})}\\
		&\lesssim \int_{0}^{T} \int_{0}^{T}  \frac{1}{|t-s|^{\frac{1}{3}}} \, dtds +T \|u_0\|^4_{L^2} < +\infty.
		\endaligned $$
		Recalling that we have proved in Lemma \ref{finite energy estimate} that $\{u_m\}_{m\geq 1} \subset L^\infty(0,T;L^2) \cap L^2(0,T;H^1)$, $\P$-a.s.,  the proof is completed by applying Theorem \ref{compact embedding} (i) and (ii).
	\end{proof}

	Thanks to the above preparations, now we can start to prove the existence of weak solutions to equation \eqref{Ito form}. We regard the sequence of Brownian motion $W:=\big\{W^{k,i}: k \in \Z_0^3, i=1,2\big\}$ as a random variable with values in $\mathcal{Y}:= C([0,T],\mathbb{C}^{\Z_0^3})$. Denote $P_m$ as the joint law of $(u_m, W)$ for each $m\geq 1$, we claim that  $\{P_m\}_{m\geq 1}$ is tight on $\mathcal{X} \times \mathcal{Y}$.
	By Prohorov theorem, there exists a subsequence $\{m_j\}_{j\geq 1}$ such that $P_{m_j}$ converges weakly to some probability measure $P$ on $\mathcal{X}\times \mathcal{Y}$ as $j \rightarrow \infty$.
	
	Applying Skorohod representation theorem, there exists a new probability space $(\tilde{\Omega},\tilde{\mathcal{F}},\tilde{\P})$, on which we can define a family $\tilde{W}^{m_j}:=\big\{\tilde{W}^{m_j,k,i}: k\in \Z_0^3, i=1,2\big\}$ of independent complex Brownian motions and a sequence of stochastic processes $\{\tilde{u}_{m_j}\}_{j\geq 1}$, together with stochastic process $\tilde{u}$ and Brownian motion $\tilde{W}$, such that
	
	\noindent (1) for any $j\geq 1$, $(\tilde{u}_{m_j},\tilde{W}^{m_j})$ has the same joint law as $(u_{m_j},W)$;
	
	\noindent (2) $\tilde{\P}$-a.s., $(\tilde{u}_{m_j},\tilde{W}^{m_j})$ converges to the limit $(\tilde{u},\tilde{W})$ in $\mathcal{X}\times \mathcal{Y}$ as $j\rightarrow \infty$.
	
	Now we turn to prove that $\big(\tilde{u},\tilde{W}\big)$ is a weak solution to equation \eqref{Ito form} with $\Lambda=1$. Since for each $j\geq 1$, $(\tilde{u}_{m_j},\tilde{W}^{m_j})$ has the same law as $(u_{m_j},W)$, then for any divergence-free vector fields $\phi \in C^\infty$, it holds
	\begin{equation}\label{existence approx}
		\begin{split}
			\<\tilde{u}_{m_j}(t),\phi\>&=\<\Pi_{m_j} u_0, \phi\>+ \int_{0}^{t} F_N(\|\tilde{u}_{m_j}(s)\|_{H^{1-\delta}})\big\<\tilde{u}_{m_j}(s) \cdot \nabla \Pi_{m_j} \phi, \tilde{u}_{m_j}(s)\big\> \, ds\\
			&\quad +\int_{0}^{t} \<\tilde{u}_{m_j}(s),\Delta \phi \> \, ds+\int_{0}^{t} \big\<\tilde{u}_{m_j}(s),S_\theta(\phi) \big\> \, ds\\
			&\quad-\sqrt{\frac{3\nu}{2}} \sum_{k,i} \theta_k \int_{0}^{t} \big\<\tilde{u}_{m_j}(s),\sigma_{k,i} \cdot \nabla\Pi_{m_j} \phi \big\> \, d\tilde{W}^{m_j,k,i}_s.
		\end{split}
	\end{equation}
	The convergence for the first term concerning the initial value is trivial.
	For $\phi \in C^\infty(\T^3,\R^3)$, we have $\Delta \phi$ is smooth, and by \cite[Lemma 5.2]{FL21 PTRF}, $S_\theta(\phi)$ is also a smooth divergence-free vector field. Hence applying the fact that $\tilde{u}_{m_j}$ converges to $\tilde{u}$ in $\mathcal{X}$ and dominated convergence theorem, we can obtain the convergence of the third and the fourth terms on the right-hand side.
	
	Before providing the proof of convergence for the other two terms, we first give a lemma, which is similar to \cite[Corollary 3.7]{FL21 PTRF}, and one can find the proof details therein.
	\begin{lemma}\label{u trajectory}
		The limit $\tilde{u}$ satisfies the following bounds $\tilde{\P}$-a.s.,
		$$\|\tilde{u}\|_{L^\infty([0,T];L^2)} \leq \|u_0\|_{L^2}, \quad \|\tilde{u}\|_{L^2([0,T];H^1)} \lesssim  \|u_0\|_{L^2}.$$
	\end{lemma}
	
	Then we turn to prove the convergence for the nonlinear term, which can be decomposed as (we omit the time parameter $s$ for convenience)
	\begin{equation}\label{nonlinear convergence J1J2}
		\begin{split}
			& \tilde{\E} \bigg[\sup_{t\in [0,T]} \bigg|\int_{0}^{t} F_N(\|\tilde{u}_{m_j}\|_{H^{1-\delta}})\<\tilde{u}_{m_j} \cdot \nabla \Pi_{m_j} \phi, \tilde{u}_{m_j}\> \, ds-\int_{0}^{t}F_N(\|\tilde{u}\|_{H^{1-\delta}})\<\tilde{u} \cdot \nabla \phi, \tilde{u}\> \, ds\bigg|\bigg]\\
			&\leq \tilde{\E} \bigg[\int_{0}^{T} F_N(\|\tilde{u}_{m_j}\|_{H^{1-\delta}}) \big|\<\tilde{u}_{m_j} \cdot \nabla \Pi_{m_j} \phi, \tilde{u}_{m_j}\> -\<\tilde{u}\cdot \nabla \phi, \tilde{u}\>\big| \, ds\bigg]\\
			&\quad +\tilde{\E} \bigg[\int_{0}^{T} \big|F_N(\|\tilde{u}_{m_j}\|_{H^{1-\delta}}) -F_N(\|\tilde{u}\|_{H^{1-\delta}})\big| \, \big|\<\tilde{u} \cdot \nabla \phi, \tilde{u}\> \big| \, ds\bigg]\\
			&=:J_1+J_2.
		\end{split}
	\end{equation}
	We start with the term $J_1$. By the definition of cut-off function, we have $F_N(\|\tilde{u}_{m_j}\|_{H^{1-\delta}})\leq 1$, and therefore it is enough to consider the quantity
	\begin{equation}\label{nonlinear decomposition}
		\begin{split}
			& \quad \ \big|\big\<\tilde{u}_{m_j} \cdot \nabla \Pi_{m_j} \phi, \tilde{u}_{m_j}\big\> -\<\tilde{u} \cdot \nabla \phi, \tilde{u}\>\big|\\
			&\leq \big|\big\<\tilde{u}_{m_j} \cdot \nabla \Pi_{m_j} \phi, \tilde{u}_{m_j}-\tilde{u}\big\>\big|+\big|\big\<(\tilde{u}_{m_j}-\tilde{u})\cdot \nabla \Pi_{m_j} \phi, \tilde{u}\big\>\big|+\big|\big\<\tilde{u} \cdot \nabla (\Pi_{m_j} \phi- \phi),\tilde{u}\big\>\big|.
		\end{split}
	\end{equation}
	We will estimate each term of \eqref{nonlinear decomposition} respectively. For $\gamma \in (0,\frac{1}{2})$, interpolation inequality yields
	$$\aligned
	\big|\big\<\tilde{u}_{m_j} \cdot \nabla \Pi_{m_j} \phi, \tilde{u}_{m_j}-\tilde{u}\big\>\big|& \leq \|\tilde{u}_{m_j} \cdot \nabla \Pi_{m_j} \phi\|_{H^{\gamma}} \|\tilde{u}_{m_j}-\tilde{u}\|_{H^{-\gamma}} \\
	&\lesssim \|\tilde{u}_{m_j}\|_{H^{\frac{1}{2}}} \|\nabla \Pi_{m_j}\phi\|_{H^{1+\gamma}} \|\tilde{u}_{m_j}-\tilde{u}\|_{H^{-\gamma}}\\
	&\leq  \|\tilde{u}_{m_j}\|_{H^1}^{\frac{1}{2}}  \|\tilde{u}_{m_j}\|_{L^2}^{\frac{1}{2}} \|\phi\|_{H^{3}} \|\tilde{u}_{m_j}-\tilde{u}\|_{H^{-\gamma}};
	\endaligned $$
	recalling Lemma \ref{finite energy estimate}, we can use H\"older's inequality to obtain
	$$\aligned
	\int_{0}^{T} \big|\big\<\tilde{u}_{m_j} \cdot \nabla \Pi_{m_j} \phi, \tilde{u}_{m_j}-\tilde{u}\big\>\big| \, ds &\lesssim \|u_0\|_{L^2}^{\frac{1}{2}} \, \|\phi\|_{H^{3}} \Big(\int_{0}^{T} \|\tilde{u}_{m_j}\|_{H^1}^2 \, ds\Big)^{\frac{1}{4}} \Big(\int_{0}^{T} \|\tilde{u}_{m_j}-\tilde{u}\|_{H^{-\gamma}}^{\frac{4}{3}} \, ds\Big)^{\frac{3}{4}}\\
	&\lesssim \|u_0\|_{L^2} \,  \|\phi\|_{H^{3}} \Big(\int_{0}^{T} \|\tilde{u}_{m_j}-\tilde{u}\|_{H^{-\gamma}}^{\frac{4}{3}} \, ds\Big)^{\frac{3}{4}},
	\endaligned $$
	which vanishes as $j\rightarrow \infty$ by the condition that $\tilde{u}_{m_j} $ converges to $\tilde{u}$ in $C([0,T];H^{-\gamma})$. Then we turn to the second term of \eqref{nonlinear decomposition}. For $\gamma \in (0,\frac{1}{2})$, we have
	$$\aligned
	\big|\big\<(\tilde{u}_{m_j}-\tilde{u})\cdot \nabla \Pi_{m_j} \phi, \tilde{u}\big\>\big| &\leq \|(\tilde{u}_{m_j}-\tilde{u})\cdot \nabla \Pi_{m_j} \phi\|_{H^{-1}} \|\tilde{u}\|_{H^{1}}\\
	&\lesssim \|\tilde{u}_{m_j}-\tilde{u}\|_{H^{-\gamma}} \|\nabla \Pi_{m_j} \phi\|_{H^{\gamma+\frac{1}{2}}}  \|\tilde{u}\|_{H^1} \\
	&\lesssim \|\tilde{u}_{m_j}-\tilde{u}\|_{H^{-\gamma}} \|\tilde{u}\|_{H^1} \|\phi\|_{H^{2}},
	\endaligned $$
	and therefore we can apply Lemma \ref{u trajectory} and Cauchy-Schwarz inequality to obtain
	$$\int_{0}^{T} \big|\big\<(\tilde{u}_{m_j}-\tilde{u})\cdot \nabla \Pi_{m_j} \phi, \tilde{u}\big\>\big| \, ds \rightarrow 0, \quad j\rightarrow \infty.$$
	For the last term of \eqref{nonlinear decomposition}, by Lemma \ref{u trajectory}, the following estimate holds:
	$$\aligned
	\big|\big\<\tilde{u} \cdot \nabla (\Pi_{m_j} \phi- \phi),\tilde{u}\big\>\big| \leq \|\tilde{u}\|_{L^2}^2 \|\nabla (\Pi_{m_j} \phi- \phi)\|_{L^\infty}\lesssim \|u_0\|_{L^2}^2 \, \|\Pi_{m_j} \phi- \phi\|_{H^{3}}.
	\endaligned $$
	Since $\phi$ is smooth, $\Pi_{m_j} \phi$ converges to $\phi$ in $H^{3}$. Taking all the discussions into consideration, we have proved that $J_1$ tends to $0$ as $j\rightarrow \infty$.
	
	Then let us turn to the term $J_2$. By Sobolev embedding theorem and Lemma \ref{u trajectory}, it holds
	$$|\<\tilde{u} \cdot \nabla \phi,\tilde{u}\>| \leq \|\tilde{u} \cdot \nabla \phi\|_{L^2} \|\tilde{u}\|_{L^2} \leq \|\tilde{u}\|_{L^2}^2 \|\nabla\phi\|_{L^\infty} \lesssim \|u_0\|_{L^2}^2 \|\phi\|_{H^{3}}.$$
	Besides, by Lemma \ref{cut-off function estimate} below and interpolation inequality, we have
	\begin{equation}\label{J2 estimate}
		\begin{split}
			\big|F_N(\|\tilde{u}_{m_j}\|_{H^{1-\delta}})-F_N(\|\tilde{u}\|_{H^{1-\delta}})\big| \lesssim_N \|\tilde{u}_{m_j}-\tilde{u}\|_{H^{1-\delta}} .
		\end{split}
	\end{equation}
	By the fact that $\tilde{u}_{m_j}$ converges to $\tilde{u}$ $\P$-\text{a.s.} in $L^2([0,T];H^{1-\delta})$, we can use H\"older's inequality to prove that $J_2$ vanishes as $j\rightarrow \infty$, thus we get the convergence for the nonlinear term.
	
	Finally, it remains to show the convergence of the stochastic term. We will follow the proof of \cite[Theorem 2.2]{FGL21 JEE} to give a sketch. Fix any $M\in \mathbb{N}$ and make the following decomposition (we omit the unimportant parameter $\sqrt{3\nu/2}$ for the simplicity of notation):
	$$\aligned
	&\quad \ \tilde{\E} \bigg[\sup_{t\in [0,T]} \bigg|\sum_{k,i} \theta_k\int_0^t \big\<\tilde{u}_{m_j}(s), \sigma_{k,i}\cdot \nabla \Pi_{m_j} \phi\big\> \, d\tilde{W}_s^{m_j,k,i}-\sum_{k,i} \theta_k\int_0^t \big\<\tilde{u}(s), \sigma_{k,i}\cdot \nabla \phi\big\> \, d\tilde{W}_s^{k,i}\bigg|\bigg]\\
	&\leq \tilde{\E} \bigg[\sup_{t\in [0,T]} \bigg|\sum_{|k|>M,i} \theta_k\int_0^t \big\<\tilde{u}_{m_j}(s), \sigma_{k,i}\cdot \nabla \Pi_{m_j} \phi\big\> \, d\tilde{W}_s^{m_j,k,i}\bigg|\bigg]\\
	&\  +\tilde{\E} \bigg[\sup_{t\in [0,T]} \bigg|\sum_{|k|>M,i} \theta_k\int_0^t \big\<\tilde{u}(s), \sigma_{k,i}\cdot \nabla \phi\big\> \, d\tilde{W}_s^{k,i}\bigg|\bigg]\\
	&\  +\tilde{\E} \bigg[\sup_{t\in [0,T]} \bigg|\sum_{|k|\leq M,i} \theta_k\bigg(\int_0^t \big\<\tilde{u}_{m_j}(s), \sigma_{k,i}\cdot \nabla \Pi_{m_j} \phi\big\> \, d\tilde{W}_s^{m_j,k,i}-\int_0^t \big\<\tilde{u}(s), \sigma_{k,i}\cdot \nabla \phi\big\> \, d\tilde{W}_s^{k,i}\bigg)\bigg|\bigg].
	\endaligned $$
	Denote $\|\theta\|_{\ell^\infty_{>M}} =\sup_{|k|>M} |\theta_k|$, which vanishes as $M\rightarrow \infty$. By Burkholder-Davis-Gundy inequality, we have
	$$\aligned
	&\quad \ \tilde{\E} \bigg[\sup_{t\in [0,T]} \bigg|\sum_{|k|>M,i} \theta_k\int_0^t \big\<\tilde{u}_{m_j}(s), \sigma_{k,i}\cdot \nabla \Pi_{m_j} \phi \big\> \, d\tilde{W}_s^{m_j,k,i}\bigg|\bigg]\\
	& \lesssim \tilde{\E} \bigg[ \bigg(\sum_{|k|>M,i} \theta_k^2 \int_0^T \big\<\tilde{u}_{m_j}(s), \sigma_{k,i}\cdot \nabla \Pi_{m_j} \phi \big\>^2 \, ds \bigg)^\frac{1}{2}\bigg]\\
	&\lesssim \|\theta\|_{\ell^\infty_{>M}} \tilde{\E} \bigg[ \bigg(\sum_{|k|>M,i}  \int_0^T \big\<\tilde{u}_{m_j}(s) \cdot \nabla \Pi_{m_j} \phi, \sigma_{k,i}  \big\>^2 \, ds \bigg)^\frac{1}{2}\bigg]\\
	&\lesssim T^{1/2} \|\theta\|_{\ell^\infty_{>M}}  \|u_0\|_{L^2} \|\nabla \phi\|_{L^\infty} \rightarrow 0, \quad M\rightarrow \infty.
	\endaligned $$
	Similar calculations hold for the second term by applying Lemma \ref{u trajectory}. Then we turn to the last term. By point (2) above  formula \eqref{existence approx}, we have $\tilde{\P}$-a.s., $\<\tilde{u}_{m_j}(s), \sigma_{k,i}\cdot \nabla \Pi_{m_j} \phi\> \rightarrow \<\tilde{u}(s), \sigma_{k,i}\cdot \nabla \phi \> $ and $\tilde{W}_s^{m_j,k,i} \rightarrow \tilde{W}_s^{k,i}$ for all $s \in [0,T]$. Hence by  \cite[Lemma 3.2]{Luo11 SPA}, we only need to verify that, for any $|k|\leq M$, it holds
	$$\bigg[\tilde{\E} \int_{0}^{T} \big\<\tilde{u}(s), \sigma_{k,i}\cdot \nabla \phi \big\>^4 \, ds \bigg] \vee \bigg[\sup_{j \geq 1} \tilde{\E} \int_{0}^{T} \big\<\tilde{u}_{m_j}(s), \sigma_{k,i}\cdot \nabla \Pi_{m_j} \phi\big\>^4 \, ds\bigg]<+\infty,$$
	which is a direct conclusion of (similar estimates can be applied for the other part)
	$$\big|\big\<\tilde{u}(s), \sigma_{k,i}\cdot \nabla \phi \big\>\big| \leq \|\tilde{u}(s)\|_{L^2} \|\sigma_{k,i} \cdot \nabla \phi\|_{L^2} \lesssim \|u_0\|_{L^2} \|\nabla\phi\|_{L^\infty}.$$
	Hence we proved the convergence of stochastic term if we let $j \rightarrow \infty$ and then $M\rightarrow \infty$. Thanks to the above all discussions, letting $j\rightarrow \infty$ in \eqref{existence approx}, it holds $\tilde{\P}$-a.s. for all $t\in [0,T]$,
	$$\aligned
	\<\tilde{u}(t),\phi\>&=\< u_0, \phi\>+ \int_{0}^{t} F_N(\|\tilde{u}(s)\|_{H^{1-\delta}})\big\<\tilde{u}(s) \cdot \nabla  \phi, \tilde{u}(s)\big\> \, ds+ \int_{0}^{t} \<\tilde{u}(s),\Delta \phi \> \, ds\\
	&\quad +\int_{0}^{t} \big\<\tilde{u}(s),S_\theta(\phi) \big\> \, ds-\sqrt{\frac{3\nu}{2}}\sum_{k,i} \theta_k \int_{0}^{t} \big\<\tilde{u}(s),\sigma_{k,i} \cdot \nabla \phi\big\> \, d\tilde{W}^{k,i}_s,
	\endaligned $$
	and we obtain the existence of the weak solution to equation \eqref{Ito form} with $\Lambda=1$.
	
	\begin{remark}\label{rmk on existence}
		Our approach to proving existence does not cover the case where $\Lambda=1$ and $\delta=0$. Recall that we have proved in Corollary \ref{tight} that the law of $\{u_m\}_m$ is tight on $L^2([0,T];H^{1-\delta})$ for $\delta \in (0,\frac{1}{4})$. As $\delta=0$, formula \eqref{J2 estimate} reads as
		$$	\big|F_N(\|\tilde{u}_{m_j}\|_{H^{1}})-F_N(\|\tilde{u}\|_{H^{1}})\big| \lesssim_N \|\tilde{u}_{m_j}-\tilde{u}\|_{H^{1}},$$
		hence $J_2$ does not vanish as $j \rightarrow \infty$ and therefore we cannot establish the convergence for the nonlinear term.
		This is why we assume that $\delta \in (0,\frac{1}{4})$ as $\Lambda=1$.
	\end{remark}
	
	\subsection{Pathwise uniqueness} \label{subsec-unique}
	
	In this section, we follow \cite{MR09} to prove the pathwise uniqueness of weak solution to \eqref{Ito form} with $\Lambda=1$. To begin with, we give a crucial lemma, which is similar to \cite[Lemma 2.1]{MR09}.
	
	\begin{lemma}\label{cut-off function estimate}
		Fix $\delta \in [0,1/4)$, then for every $u,  v \in H^{1-\delta}$ and each $N>0$, it holds
		\begin{equation}\label{estimate on F_N}
			\big|F_N(\|u\|_{H^{1-\delta}})-F_N(\|v\|_{H^{1-\delta}})\big| \leq \frac{1}{N} F_N(\|u\|_{H^{1-\delta}}) F_N(\|v\|_{H^{1-\delta}}) \|u-v\|_{H^{1-\delta}}.
		\end{equation}
	\end{lemma}
	\begin{proof}
		Fixing an arbitrary $N>0$, we discuss in the following three cases.
		
		\noindent\emph{Case 1.} $\|u\|_{H^{1-\delta}} \leq N$ and $\|v\|_{H^{1-\delta}} \leq N$. Then $F_N(\|u\|_{H^{1-\delta}})=F_N(\|v\|_{H^{1-\delta}})=1$, and the conclusion is trivial.
		
		\noindent\emph{Case 2.} $\|u\|_{H^{1-\delta}} > N$ and $\|v\|_{H^{1-\delta}} > N$. Now we have
		$$\aligned \big|F_N(\|u\|_{H^{1-\delta}})-F_N(\|v\|_{H^{1-\delta}})\big|&=\frac{N\big|\|v\|_{H^{1-\delta}}-\|u\|_{H^{1-\delta}}\big|}{\|u\|_{H^{1-\delta}}\|v\|_{H^{1-\delta}}}
		\leq \frac{N \|u-v\|_{H^{1-\delta}}}{\|u\|_{H^{1-\delta}}\|v\|_{H^{1-\delta}}}\\
		&=\frac{1}{N} F_N(\|u\|_{H^{1-\delta}}) F_N(\|v\|_{H^{1-\delta}}) \|u-v\|_{H^{1-\delta}}.
		\endaligned $$
		
		\noindent\emph{Case 3.} $\|u\|_{H^{1-\delta}} \leq N$ and $\|v\|_{H^{1-\delta}} > N$. By the definition of cut-off function, it holds
		$$\aligned
		\big|F_N(\|u\|_{H^{1-\delta}})-F_N(\|v\|_{H^{1-\delta}})\big|&=\frac{\|v\|_{H^{1-\delta}}-N}{\|v\|_{H^{1-\delta}}}
		\leq \frac{\|v\|_{H^{1-\delta}}-\|u\|_{H^{1-\delta}}}{\|v\|_{H^{1-\delta}}} \\
		&\leq \frac{\|u-v\|_{H^{1-\delta}}}{\|v\|_{H^{1-\delta}}}=\frac{1}{N} F_N(\|u\|_{H^{1-\delta}}) F_N(\|v\|_{H^{1-\delta}}) \|u-v\|_{H^{1-\delta}}.
		\endaligned $$
		Since $u$ and $v$ are symmetric, the case for $\|v\|_{H^{1-\delta}} \leq N$ and $\|u\|_{H^{1-\delta}} > N$ is exactly the same as Case 3, so we omit the duplicate proof.
	\end{proof}

	Next, we present another estimate which will be used in the sequel.
	\begin{lemma}\label{Buvw estimate}
		For any $u,v,w \in H^1$ and $\delta \in [0,1/4)$, we have
		$$\big|\big\< w,  \Pi(u\cdot \nabla v)\big\>\big| \lesssim \|v\|_{H^{1-\delta}} \|u\|_{H^{1-\delta}} \|w\|_{H^1}^{2\delta+\frac{1}{2}} \|w\|_{L^2}^{\frac{1}{2}-2\delta}.$$
	\end{lemma}
	
	\begin{proof}
		By the divergence-free property of $u$ and $w$, we use \cite[Corollary 2.55]{HB Fourier analysis} and interpolation inequality to get
		$$\aligned
		\big|\big\< w,  \Pi(u\cdot \nabla v)\big\>\big|&=|\<w, u\cdot \nabla v\>| =|\<v, u\cdot \nabla w\>|\\
		&\leq \|v\|_{H^{1-\delta}} \|u\cdot\nabla w\|_{H^{\delta-1}}\\
		&\lesssim \|v\|_{H^{1-\delta}}  \|u\|_{H^{1-\delta}} \|\nabla w\|_{H^{2\delta-\frac{1}{2}}} \\
		&\lesssim \|v\|_{H^{1-\delta}} \|u\|_{H^{1-\delta}} \|w\|_{H^1}^{2\delta+\frac{1}{2}} \|w\|_{L^2}^{\frac{1}{2}-2\delta},
		\endaligned $$
		which implies the conclusion of this lemma.
	\end{proof}
	
	Based on the above preparations, now we can provide the proof of pathwise uniqueness. Suppose $u$, $v$ are both the solutions of \eqref{Ito form} with the same initial value $u_0$ and denote $w:=u-v$. We aim to prove that $w(t)\equiv 0$ for any $t\in [0,T]$, so we will make energy estimate for $w$. To apply It\^o formula, we need to verify the condition in \cite[Theorem 2.13]{RL 1990}, i.e.,
	\begin{equation}\label{ito verification}
		\int_{0}^{T}  \|w\|_{H^1}^2+\|\partial_t w \|_{H^{-1}}^2 dt<+\infty, \quad \P \text{-a.s.}.
	\end{equation}
	The first part is trivial due to the fact that $u, v \in  L^2([0,T];H^1)$. For the second part, we only discuss the nonlinear term, i.e., we will prove that
	$$\int_{0}^{T} \big\| F_N(\|u\|_{H^{1-\delta}}) \Pi(u\cdot \nabla u) -F_N(\|v\|_{H^{1-\delta}}) \Pi(v\cdot \nabla v) \big\|_{H^{-1}}^2 \, dt < +\infty.$$
	Since $u$ is divergence-free, for $\delta \in (0,\frac{1}{4})$, by H\"older's inequality, Sobolev embedding theorem and interpolation inequality, it holds
	$$\big\|\Pi(u\cdot \nabla u) \big\|_{H^{-1}} \lesssim \|u^2\|_{L^2} \leq \|u\|_{L^4}^2 \lesssim \|u\|_{H^{\frac{3}{4}}}^2 \leq  \|u\|_{H^{1-\delta}}^{2}.$$
	Furthermore, by the definition of $F_N(\|u\|_{H^{1-\delta}})$, we have $F_N(\|u\|_{H^{1-\delta}}) \|u\|_{H^{1-\delta}} \leq N$, hence
	\begin{equation}\nonumber
		\begin{split}
			\int_{0}^{T} F_N^2(\|u\|_{H^{1-\delta}}) \big\|\Pi(u\cdot \nabla u) \big\|_{H^{-1}}^2 \, dt\lesssim \int_{0}^{T} F_N^2(\|u\|_{H^{1-\delta}})  \|u\|_{H^{1-\delta}}^4 \, dt\lesssim  N^2 \|u_0\|_{L^2}^2< +\infty.
		\end{split}
	\end{equation}
	Similar proofs hold for the other term associated to $v$. As \eqref{ito verification} has been verified, we can apply It\^o formula for $w$ to arrive at
	\begin{equation}\label{energy estimate}
		\begin{split}
			d\|w\|_{L^2}^2&=2\<w,dw\>+\<dw,dw\>\\
			&=-2\big\<w, F_N(\|u\|_{H^{1-\delta}}) \Pi(u\cdot \nabla u) -F_N(\|v\|_{H^{1-\delta}}) \Pi(v\cdot \nabla v) \big\> dt +2 \<w, \Delta w\>dt\\
			&\ +2\sqrt{\frac{3\nu}{2}} \sum_{k,i} \theta_k \big\<w,  \Pi(\sigma_{k,i} \cdot \nabla w) \big\> dW_t^{k,i} + 2\<w ,S_\theta(w)\>dt+3\nu \sum_{k,i} \theta_k^2 \big\|\Pi(\sigma_{k,i} \cdot \nabla w)\big\|_{L^2}^2 \, dt\\
			&=-2\big\<w, F_N(\|u\|_{H^{1-\delta}}) \Pi(u\cdot \nabla u) -F_N(\|v\|_{H^{1-\delta}}) \Pi(v\cdot \nabla v) \big\> dt -2 \|\nabla w\|_{L^2}^2 \, dt,
		\end{split}
	\end{equation}
	where in the last step we used the facts that $\text{div} \, \sigma_{k,i}=0$ and
	$$\<w,S_\theta(w)\>=-\frac{3\nu}{2}\sum_{k,i} \theta_k^2  \big\|\Pi(\sigma_{k,i} \cdot \nabla w)\big\|_{L^2}^2.$$
	Now we focus on the nonlinear term of \eqref{energy estimate}. By the triangle inequality, it holds
	$$\aligned
	&\quad \  F_N(\|u\|_{H^{1-\delta}}) \Pi(u\cdot \nabla u) -F_N(\|v\|_{H^{1-\delta}}) \Pi(v\cdot \nabla v)\\
	&=F_N(\|u\|_{H^{1-\delta}})  \Pi(w\cdot \nabla u) + \big[F_N(\|u\|_{H^{1-\delta}})-F_N(\|v\|_{H^{1-\delta}})\big]  \Pi(v \cdot \nabla u)+F_N(\|v\|_{H^{1-\delta}}) \Pi( v\cdot \nabla w).
	\endaligned $$
	For the first term, Lemma \ref{Buvw estimate} and interpolation inequality yield
	$$\aligned
	\big|\big\<w, F_N(\|u\|_{H^{1-\delta}})  \Pi(w\cdot \nabla u) \big\> \big|&\lesssim  F_N(\|u\|_{H^{1-\delta}}) \|u\|_{H^{1-\delta}} \|w\|_{H^{1-\delta}} \|w\|_{H^1}^{2\delta+\frac{1}{2}} \|w\|_{L^2}^{\frac{1}{2}-2\delta} \\
	&\lesssim N \|w\|_{H^1}^{1-\delta} \|w\|_{L^2}^\delta \|w\|_{H^1}^{2\delta+\frac{1}{2}} \|w\|_{L^2}^{\frac{1}{2}-2\delta}\\
	&=N \|w\|_{H^1}^{\delta+\frac{3}{2}} \|w\|_{L^2}^{\frac{1}{2}-\delta}.
	\endaligned $$
	Similarly, we have the same estimate for the third term:
	$$\big|\big\<w,F_N(\|v\|_{H^{1-\delta}}) \Pi( v\cdot \nabla w) \big\>\big| \lesssim N \|w\|_{H^1}^{\delta+\frac{3}{2}} \|w\|_{L^2}^{\frac{1}{2}-\delta}.$$
	For the other term, we apply  Lemmas \ref{cut-off function estimate}, \ref{Buvw estimate} and interpolation inequality to obtain
	$$\aligned
	&\quad \big|\big\<w, \big[F_N(\|u\|_{H^{1-\delta}})-F_N(\|v\|_{H^{1-\delta}})\big]  \Pi(v \cdot \nabla u) \big\>\big| \\
	&\lesssim \big|F_N(\|u\|_{H^{1-\delta}})-F_N(\|v\|_{H^{1-\delta}})\big| \, \|u\|_{H^{1-\delta}} \|v\|_{H^{1-\delta}}  \|w\|_{H^1}^{2\delta+\frac{1}{2}} \|w\|_{L^2}^{\frac{1}{2}-2\delta}\\
	&\lesssim  N \|w\|_{H^{1}}^{1-\delta} \|w\|_{L^2}^\delta \|w\|_{H^1}^{2\delta+\frac{1}{2}} \|w\|_{L^2}^{\frac{1}{2}-2\delta}\\
	&=N \|w\|_{H^1}^{\delta+\frac{3}{2}} \|w\|_{L^2}^{\frac{1}{2}-\delta}.
	\endaligned $$
	Combining the above estimates, we use Young's inequality to get
	\begin{equation}\label{nonlinear difference estimate}
		\begin{aligned}
			\big|\big\<w,F_N(\|u\|_{H^{1-\delta}})  \Pi(u\cdot \nabla u) -F_N(\|v\|_{H^{1-\delta}}) \Pi(v\cdot \nabla v) \big\>\big|
			\lesssim  \|w\|_{H^1}^2+C(\delta,N) \|w\|_{L^2}^2,
		\end{aligned}
	\end{equation}
	where $C(\delta,N)>0$ is a finite constant. Inserting the above estimate into \eqref{energy estimate} yields
	$$d \, \|w\|_{L^2}^2 \lesssim 2 C(\delta,N) \|w\|_{L^2}^2 \, dt.$$
	Applying Gr\"onwall's inequality, we have $w(t)\equiv 0$ for all $t\in[0,T]$, which implies  the pathwise uniqueness of the solution to equation \eqref{Ito form} with $\Lambda=1$.

	\section{Convergence}\label{sec-convergence}
	In Section \ref{subsec-scaling limit}, we establish the scaling limit result. For the sake of simplicity, we focus on the stochastic GMNSE in the proof, as the arguments for the hyperviscous stochastic GMNSE are similar. In Section \ref{subsec-quantitative conv}, we provide an explicit quantitative convergence rate in the $L^2([0,T];H^{1-\delta})$-norm. Again, we primarily consider the stochastic GMNSE for simplicity, while for the hyperviscous stochastic GMNSE, an additional constraint must be imposed. The details can be found in Remark \ref{rmk on quant conv rate}.

	\subsection{Scaling limit}\label{subsec-scaling limit}
	
	By Theorem \ref{well-posedness thm}, equation \eqref{sequence of thetaR} has a unique solution $u^n$ for each $n \geq 1$, such that for any divergence-free vector field $\phi \in C^\infty(\T^3,\R^3)$, it holds
	$$\aligned
	\<u^n(t),\phi\>&=\< u_0, \phi\>+ \int_{0}^{t} F_N(\|u^n(s)\|_{H^{1-\delta}}) \big\<u^n(s) \cdot \nabla  \phi, u^n(s) \big\> \, ds+ \int_{0}^{t} \<u^n(s),\Delta \phi \> \, ds\\
	&\quad +\int_{0}^{t} \big\<u^n(s),S_{\theta^n}(\phi) \big\> \, ds-\sqrt{\frac{3\nu}{2}}\sum_{k,i} \theta_k^n \int_{0}^{t} \big\<u^n(s),\sigma_{k,i} \cdot \nabla \phi\big\> \, dW^{k,i}_s.
	\endaligned $$
	If we denote $\eta^n$ as the law of $u^n$ and $Q^n$ as the joint law of $(u^n,W)$ for each $n\geq 1$, we can apply similar arguments as those in Section \ref{subsec-exist} to prove that $\{\eta^n\}_{n\geq 1}$ is tight on $\mathcal{X}$, and therefore
	$\{Q^n\}_{n\geq 1}$ is tight on $\mathcal{X} \times \mathcal{Y}$. Furthermore, there exists a subsequence $\{n_j\}_{j\geq 1}$, such that $Q^{n_j}$ converges weakly to some probability measure $Q$ on $\mathcal{X} \times \mathcal{Y}$. Also, we can find a new probability space $(\bar{\Omega},\bar{\mathcal{F}},\bar{\P})$, together with stochastic processes $\{(\bar{u}^{n_j},\bar{W}^{n_j})\}_{j\geq 1}$ and $(\bar{u},\bar{W})$ defined on $\bar \Omega$, satisfying\\
	(1') for any $j\geq 1$, $(\bar{u}^{n_j},\bar{W}^{n_j})$ has the same joint law as $(u^{n_j},W)$;\\
	(2') $\bar{\P}$-a.s., $(\bar{u}^{n_j},\bar{W}^{n_j})$ converges to $(\bar{u},\bar{W})$ in $\mathcal{X} \times \mathcal{Y}$ as $j \rightarrow \infty$.
	
	Now we can provide
	\begin{proof}[Proof of Theorem \ref{scaling limit}.]
		We first prove the uniqueness of the weak solution to the limit equation \eqref{limit eq} with $\Lambda=1$. Suppose $\bar{u}_1$ and $\bar{u}_2$ are two solutions of \eqref{limit eq} with the same initial value and denote $\xi:=\bar{u}_1-\bar{u}_2$, then it holds
		$$	\partial_t \xi+F_N(\|\bar{u}_1\|_{H^{1-\delta}}) \Pi (\bar{u}_1 \cdot \nabla \bar{u}_1)-F_N(\|\bar{u}_2\|_{H^{1-\delta}}) \Pi (\bar{u}_2 \cdot \nabla \bar{u}_2)=\Big(1+\frac{3\nu}{5}\Big)\Delta \xi.$$
		Thanks to the cut-off, we can make energy estimate for the above equation and get
		$$\aligned
		\frac{1}{2}\frac{d}{dt} \|\xi\|_{L^2}^2&=\big\<\xi, -F_N(\|\bar{u}_1\|_{H^{1-\delta}}) \Pi (\bar{u}_1 \cdot \nabla \bar{u}_1)+F_N(\|\bar{u}_2\|_{H^{1-\delta}}) \Pi (\bar{u}_2 \cdot \nabla \bar{u}_2)\big\>-\Big(1+\frac{3\nu}{5}\Big) \|\nabla \xi\|_{L^2}^2\\
		&\lesssim C(\delta,N,\nu) \|\xi\|_{L^2}^2,
		\endaligned $$
		where in the last step we utilized the calculations similar to \eqref{nonlinear difference estimate}  but we changed the weights when using Young's inequality. Applying Gr\"onwall's lemma, we get $\xi \equiv 0$. Meanwhile, we can easily show that $\bar{u} \in L^\infty([0,T];L^2)\cap L^2([0,T];H^1)$.
		
		Next, we prove that the limit process $\bar{u}$ mentioned in the point (2') above solves the deterministic 3D GMNSE \eqref{limit eq} with $\Lambda=1$ in the sense that for any divergence-free vector fields $\phi \in C^\infty(\T^3,\R^3)$, it holds
		$$\aligned
		\<\bar{u}(t), \phi\>-\<u_0, \phi\>=\int_{0}^{t} F_N(\|\bar{u}(s)\|_{H^{1-\delta}}) \big\<\bar{u}(s) \cdot \nabla\phi, \bar{u}(s) \big\> \, ds+\Big(1+\frac{3\nu}{5}\Big) \int_{0}^{t} \big\<\bar{u}(s), \Delta\phi \big\> \, ds.
		\endaligned $$
		Since $(\bar{u}^{n_j},\bar{W}^{n_j})$ has the same law as $(u^{n_j},W)$, hence it satisfies the following identity:
		$$\aligned
		\<\bar{u}^{n_j}(t),\phi\>&=\< u_0, \phi\>+ \int_{0}^{t} F_N(\|\bar{u}^{n_j}(s)\|_{H^{1-\delta}}) \big\<\bar{u}^{n_j}(s) \cdot \nabla  \phi, \bar{u}^{n_j}(s) \big\> \, ds+ \int_{0}^{t} \<\bar{u}^{n_j}(s),\Delta \phi \> \, ds\\
		&\quad +\int_{0}^{t} \big\<\bar{u}^{n_j}(s),S_{\theta^{n_j}}(\phi) \big\> \, ds-\sqrt{\frac{3\nu}{2}}\sum_{k,i} \theta_k^{n_j} \int_{0}^{t} \big\<\bar{u}^{n_j}(s),\sigma_{k,i} \cdot \nabla \phi \big\> \, dW^{n_j,k,i}_s.
		\endaligned $$
		The convergence for the third term on the right-hand side is trivial as $\phi$ is smooth. For the convergence of the nonlinear term, we can use similar proofs as those below Lemma \ref{u trajectory} to obtain, hence we omit the calculations here and we focus on the latter two terms. By the triangle inequality, it holds
		$$\Big|\big\<\bar{u}^{n_j}, S_{\theta^{n_j}}(\phi) \big\> -\big\<\bar{u}, \frac{3\nu}{5}\Delta \phi \big\> \Big| \leq \Big|\big\<\bar{u}^{n_j}-\bar{u},S_{\theta^{n_j}}(\phi) \big\>\Big|+\Big|\big\<\bar{u},S_{\theta^{n_j}}(\phi) -\frac{3\nu}{5}\Delta \phi \big\>\Big|.$$
		By Theorem \ref{S limit thm}, the following convergence holds in $L^2$ for $\phi \in C^\infty(\T^3,\R^3)$:
		$$\lim_{j\rightarrow \infty} S_{\theta^{n_j}}(\phi)=\frac{3\nu}{5} \Delta \phi,$$
		hence by the condition that $\bar{u} \in L^\infty ([0,T];L^2)$, the second term vanishes. Besides, notice that $\bar{u}^{n_j} \rightarrow \bar{u}$ in $L^2([0,T];H^{1-\delta})$, we can apply Cauchy-Schwarz inequality to prove that, $\P$-a.s.,
		$$\int_{0}^{t} \Big|\big\<\bar{u}^{n_j}(s)-\bar{u}(s),S_{\theta^{n_j}}(\phi) \big\>\Big| \, ds \leq T^\frac{1}{2} \|S_{\theta^{n_j}}(\phi) \|_{L^2}  \|\bar{u}^{n_j}-\bar{u}\|_{L^2([0,T];H^{1-\delta})} \rightarrow 0, \quad j \rightarrow \infty.$$
		Combining  the above arguments, we obtain the convergence for this term.
		
		Then we will show that the martingale part vanishes as $j\rightarrow \infty$. By It\^o isometry, it holds
		$$\aligned
		\bar{\E} \bigg(\sqrt{\frac{3\nu}{2}}\sum_{k,i} \theta_k^{n_j} \int_{0}^{t} \<\bar{u}^{n_j}(s),\sigma_{k,i} \cdot \nabla \phi\> \, dW^{n_j,k,i}_s\bigg)^{\! 2}
		&=\frac{3\nu}{2}\sum_{k,i} \big(\theta_k^{n_j}\big)^2 \bar{\E} \int_{0}^{t} \big|\<\bar{u}^{n_j}(s),\sigma_{k,i} \cdot \nabla \phi\>\big|^2 ds\\
		&\lesssim \nu \|\theta^{n_j}\|_{\ell^\infty}^2 \bar{\E} \int_{0}^{t} \sum_{k,i}  \big|\<\bar{u}^{n_j}(s)\cdot \nabla \phi, \sigma_{k,i} \>\big|^2 ds.
		\endaligned $$
		Furthermore, we have
		$$ \sum_{k,i}  \big|\<\bar{u}^{n_j}(s)\cdot \nabla \phi,\sigma_{k,i} \>\big|^2 \leq \big\| \bar{u}^{n_j}(s) \cdot \nabla \phi \big\|_{L^2}^2 \leq  \|\nabla \phi\|_{L^\infty}^2 \big\|\bar{u}^{n_j}(s)\big\|_{L^2}^2,$$
		which can be combined with the above estimate to yield
		$$\aligned
		\bar{\E} \bigg(\sqrt{\frac{3\nu}{2}}\sum_{k,i} \theta_k^{n_j} \int_{0}^{t} \big\<\bar{u}^{n_j}(s),\sigma_{k,i} \cdot \nabla \phi \big\> \, dW^{n_j,k,i}_s\bigg)^{\! 2}
		&\lesssim  \nu \|\theta^{n_j}\|_{\ell^\infty}^2 \|\nabla \phi\|_{L^\infty}^2  \bar{\E} \int_{0}^{t} \big\|\bar{u}^{n_j}(s) \big\|_{L^2}^2 \, ds\\
		&\lesssim \nu T \|\theta^{n_j}\|_{\ell^\infty}^2 \|\nabla \phi\|_{L^\infty}^2 \|u_0\|_{L^2}^2.
		\endaligned $$
		By the assumptions \eqref{assumption on theta} and $u_0 \in L^2$, we proved that this term tends to $0$ in the sense of mean square. Therefore, we have verified that $\bar{u}$ is the unique weak solution to the deterministic GMNSE \eqref{limit eq} with $\Lambda=1$.
		
		Finally, we turn to present the strong convergence estimate. Note that $\{u^n\}_{n\geq 1}$ is a sequence of variables defined on the same probability space and they converge to a deterministic limit, hence convergence in law is equivalent to convergence in probability. Since  we have proved that $\eta^n$, the law of $u^n$, converges to the Dirac measure $\delta_{\bar{u}_\cdot}$ on $\mathcal{X}$, it follows that for any $\epsilon>0$,
		$$\P \big\{\|u^n-\bar{u}\|_{\mathcal{X}} > \epsilon\big\} \rightarrow 0, \quad n \rightarrow \infty,$$
		where	$\| \cdot \|_{\mathcal{X}}$ is defined above Theorem \ref{scaling limit}. Furthermore,  if we apply similar arguments to the proofs of Lemma \ref{finite energy estimate}, we have
		$$\sup_{t\in [0,T]} \sup_{n\geq 1}  \bigg[\|u^n(t)\|_{L^2}^2+2 \int_0^t \|\nabla u^n(s)\|_{L^2}^2 \, ds\bigg] \leq \|u_0\|_{L^2}^2, \quad \P \text{-a.s.}.$$
		Hence for $p<+\infty$, the following moment estimate holds
		$$\aligned
		\E \Big[\|u^n-\bar{u}\|_{\mathcal{X}}^p\Big] &= \E \Big[\|u^n-\bar{u}\|_{\mathcal{X}}^p \, \textbf{1}_{\{\|u^n-\bar{u}\|_{\mathcal{X}}>\epsilon\}}\Big]+ \E \Big[\|u^n-\bar{u}\|_{\mathcal{X}}^p \, \textbf{1}_{\{\|u^n-\bar{u}\|_{\mathcal{X}}\leq \epsilon\}}\Big] \\
		&\lesssim \|u_0\|_{L^2}^p \, \P\big\{\|u^n-\bar{u}\|_{\mathcal{X}} > \epsilon\big\}+\epsilon^p.
		\endaligned $$
		Noticing that $\epsilon>0$ is arbitrary, the above estimate vanishes as $n \rightarrow \infty$.
	\end{proof}

	\subsection{Quantitative convergence rate}\label{subsec-quantitative conv}
	As mentioned at the beginning of Section \ref{sec-convergence}, we will mainly discuss the stochastic GMNSE in this section, i.e., we always take $\Lambda=1$ unless mentioned otherwise. First of all, we present the following result which is required in subsequent proofs.

	\begin{lemma}\label{CH1 bounded for baru}
		For deterministic 3D GMNSE, if the initial value $u_0 \in H^1$, then $\bar{u} \in L^\infty([0,T];H^1)$. More precisely, it holds $\sup_{t\in [0,T]} \|\bar{u}_t\|_{H^1} \lesssim_{\nu,\delta,N,T} \|u_0\|_{H^1}$.
	\end{lemma}
	
	\begin{proof}
		For equation \eqref{limit eq} with $\Lambda=1$, we make the following formal calculation:
		$$\aligned
		\frac{d}{dt} \|\bar{u}\|_{H^1}^2&=\frac{d}{dt} \big\|(-\Delta)^{\frac{1}{2}} \bar{u} \big\|_{L^2}^2=2\big\<(-\Delta)^{\frac{1}{2}} \bar{u},(-\Delta)^{\frac{1}{2}} \partial_t\bar{u}\big\>\\
		&=2 \, \Big\<-\Delta \bar{u}, -F_N(\|\bar{u}\|_{H^{1-\delta}}) \Pi(\bar{u}\cdot \nabla \bar{u})+\Big(1+\frac{3\nu }{5}\Big)\Delta \bar{u}\Big\>\\
		&=2F_N(\|\bar{u}\|_{H^{1-\delta}})  \<\Delta \bar{u}, \bar{u} \cdot\nabla \bar{u}\>-2\Big(1+\frac{3\nu }{5}\Big) \|\Delta \bar{u}\|_{L^2}^2.
		\endaligned $$
		By H\"older's inequality, Sobolev embedding and interpolation inequality, it holds
		$$\aligned
		\big|\<\Delta \bar{u}, \bar{u} \cdot\nabla \bar{u}\>\big| &\leq \|\Delta \bar{u}\|_{L^2} \|\bar{u}\|_{L^\frac{6}{1+2\delta}} \|\nabla \bar{u}\|_{L^\frac{3}{1-\delta}} \\
		&\lesssim \|\Delta \bar{u}\|_{L^2}  \|\bar{u}\|_{H^{1-\delta}} \|\nabla \bar{u}\|_{H^{\frac{1}{2}+ \delta}}\\
		&\lesssim \|\Delta \bar{u}\|_{L^2}  \|\bar{u}\|_{H^{1-\delta}} \|\bar{u}\|_{H^1} ^{\frac{1}{2}- \delta} \|\bar{u}\|_{H^2}^{\frac{1}{2}+ \delta}.
		\endaligned $$
		By the definition of cut-off function and Young's inequality, we have
		$$F_N(\|\bar{u}\|_{H^{1-\delta}}) \big| \<\Delta \bar{u}, \bar{u} \cdot\nabla \bar{u}\>\big| \lesssim N \|\Delta \bar{u}\|_{L^2}^{\frac{3}{2}+ \delta} \|\bar{u}\|_{H^1} ^{\frac{1}{2}- \delta} \leq \|\Delta \bar{u}\|_{L^2}^2+C(N, \delta) \|\bar{u}\|_{H^1} ^2,$$
		where $C(N, \delta)>0$ is a finite constant.
		Summarizing the above estimates, we arrive at
		$$\frac{d}{dt} \|\bar{u}\|_{H^1}^2 \leq 2C(N, \delta) \|\bar{u}\|_{H^1} ^2,$$
		then Gr\"onwall's inequality yields our desired estimate.
	\end{proof}

Now we can state the main result of this part.
	
	\begin{theorem}\label{quantitative convergence thm}
		Let $u^n$ and $\bar{u}$ be the solutions of equations \eqref{sequence of thetaR} and \eqref{limit eq} with the same initial value $u_0\in H^1$, respectively.  Given any $N>0$, we can find $\nu \gg 1$ satisfying
		$$ \nu^{-2} (1+N^2) \big(1+\|u_0\|_{H^1}^2\big) \ll 1,$$
		such that for any fixed $T<+\infty$ and $\delta \in (0,1/4)$, the following estimate holds:
		$$\E \bigg[	\int_{0}^{T} \|u_s^n-\bar{u}_s\|_{H^{1-\delta}}^2 \, ds\bigg] \lesssim 	\nu^2 n^{-2\delta} N^{-2} +\nu^{\frac{\delta}{2}+2} \delta^{\frac{\delta-5}{5}} T^{\frac{7\delta-\delta^2}{10}} N^{-2} \|\theta^n\|_{\ell^\infty}^{\frac{2\delta}{5}}.$$
	\end{theorem}
	
	\begin{remark}\label{rmk on quantitative conv}
		1. Fix $N$ and $\|u_0\|_{H^1}$, we first choose a large $\nu$ satisfying our assumption; then letting $n\rightarrow \infty$, by the condition \eqref{assumption on theta}, the above estimate can be sufficiently small for any fixed $T<+\infty$ and $\delta \in (0,\frac{1}{4})$.
		
		2. We suppose $u_0 \in H^1$ just for the sake of estimating $\sup_{t\in [0,T]} \|\bar{u}(t)\|_{H^1}$ above, which will be used for example in \eqref{baru H1 norm}. Actually, for the sequence of stochastic GMNSEs \eqref{sequence of thetaR} with $\Lambda=1$, it is enough to assume $u_0 \in L^2$, which is a direct consequence of $u_0 \in H^1$.
	\end{remark}

	\begin{proof}[Proof of Theorem \ref{quantitative convergence thm}]
		By the scaling limit result in Theorem \ref{scaling limit}, we can rewrite \eqref{sequence of thetaR} with $\Lambda=1$ as
		$$\aligned
		du^n+F_N(\|u^n\|_{H^{1-\delta}}) \Pi(u^n\cdot \nabla u^n) \, dt=&\Big(1+\frac{3\nu}{5}\Big) \Delta u^n \, dt+\sqrt{\frac{3\nu}{2}}\sum_{k,i}  \theta_k^n \Pi(\sigma_{k,i} \cdot \nabla u^n) \, dW_t^{k,i}\\
		&+\Big(S_{\theta^n}(u^n)-\frac{3\nu}{5} \Delta u^n\Big) \, dt.
		\endaligned$$
		If we denote $P_t=e^{(1+\frac{3\nu}{5}) t \Delta}$ for any $t\geq 0$, then we get the mild formulation to the above equation as
		\begin{equation}\label{u^R mild form}
			u^n_t=P_t u_0-\int_{0}^{t} P_{t-s} \Big(F_N(\|u^n_s\|_{H^{1-\delta}}) \Pi(u^n_s\cdot \nabla u^n_s) \Big)\, ds+\int_{0}^{t} P_{t-s} \Big(S_{\theta^n}(u_s^n)-\frac{3\nu}{5} \Delta u_s^n \Big) \, ds+ Z_t^n,
		\end{equation}
		where $Z_t^n$ is a stochastic convolution with the following expression:
		$$Z_t^n:=\sqrt{\frac{3\nu}{2}}\sum_{k,i}\theta_k^n \int_{0}^{t} P_{t-s}  \Pi\big(\sigma_{k,i} \cdot \nabla u_s^n\big) \, dW_s^{k,i}.$$
		Besides, the equation \eqref{limit eq} with $\Lambda=1$ can also be written in mild form:
		\begin{equation}\label{bar u mild form}
			\bar{u}_t=P_t u_0- \int_{0}^{t} P_{t-s} \Big(F_N(\|\bar{u}_s\|_{H^{1-\delta}}) \Pi(\bar{u}_s\cdot \nabla \bar{u}_s) \Big)\, ds.
		\end{equation}
		So we make difference between \eqref{u^R mild form} and \eqref{bar u mild form} to get
		$$\aligned
		u_t^n-\bar{u}_t= &-\int_{0}^{t} P_{t-s} \Big(F_N(\|u^n_s\|_{H^{1-\delta}}) \Pi(u^n_s\cdot \nabla u^n_s)-F_N(\|\bar{u}_s\|_{H^{1-\delta}}) \Pi(\bar{u}_s\cdot \nabla \bar{u}_s) \Big) \, ds\\
		&+\int_{0}^{t} P_{t-s}\Big(S_{\theta^n}(u_s^n)-\frac{3\nu}{5} \Delta u_s^n \Big) \, ds+ Z_t^n.
		\endaligned $$
		Taking $H^{1-\delta}$-norm and then square on the both sides yields
		\begin{equation}\label{H-lambda norm for difference}
			\begin{split}
				\|u_t^n-\bar{u}_t\|^2_{H^{1-\delta}} &\lesssim \bigg\| \int_{0}^{t} P_{t-s} \Big(F_N(\|u^n_s\|_{H^{1-\delta}}) \Pi(u^n_s\cdot \nabla u^n_s)-F_N(\|\bar{u}_s\|_{H^{1-\delta}}) \Pi(\bar{u}_s\cdot \nabla \bar{u}_s) \Big) \, ds\bigg\|_{H^{1-\delta}}^2\\
				&\quad+\bigg\|\int_{0}^{t}P_{t-s}\Big(S_{\theta^n}(u_s^n)-\frac{3\nu}{5} \Delta u_s^n \Big) \, ds \bigg\|^2_{H^{1-\delta}}+\big\|Z_t^n\big\|^2_{H^{1-\delta}}\\
				&=:A_{1,n}(t)+A_{2,n}(t)+\big\|Z_t^n\big\|^2_{H^{1-\delta}}.
			\end{split}
		\end{equation}
		For $A_{1,n}(t)$, we can use the triangle inequality to further decompose it as
		$$\aligned
		A_{1,n}(t)
		&\lesssim \bigg\| \int_{0}^{t} P_{t-s} \Big[F_N(\|u^n_s\|_{H^{1-\delta}})  \Pi \big(u^n_s\cdot \nabla (u^n_s-\bar{u}_s)\big) \Big]\, ds\bigg\|_{H^{1-\delta}}^2\\
		&\quad+\bigg\| \int_{0}^{t} P_{t-s} \Big[F_N(\|u_s^n\|_{H^{1-\delta}}) \Pi \big( (u^n_s-\bar{u}_s)\cdot \nabla \bar{u}_s \big) \Big]\, ds \bigg\|_{H^{1-\delta}}^2\\
		&\quad+\bigg\| \int_{0}^{t} P_{t-s} \Big[\big(F_N(\|u_s^n\|_{H^{1-\delta}})-F_N(\|\bar{u}_s\|_{H^{1-\delta}})\big) \Pi \big( \bar{u}_s \cdot \nabla \bar{u}_s \big) \Big]\, ds \bigg\|_{H^{1-\delta}}^2\\
		&=:A_{1,1,n}(t)+A_{1,2,n}(t)+A_{1,3,n}(t).
		\endaligned $$
		Let us consider the time integral of each term separately. By Lemma \ref{semigroup property},  for any $\delta \in (0,\frac{1}{4})$,  it holds
		$$\aligned
		\int_{0}^{t} A_{1,1,n}(s) \, ds &\lesssim \Big(1+\frac{3\nu}{5}\Big)^{-2} \int_{0}^{t} \Big\| F_N(\|u^n_s\|_{H^{1-\delta}}) \Pi \big(u^n_s\cdot \nabla (u^n_s-\bar{u}_s)\big)\Big\|_{H^{-1-\delta}}^2 \, ds\\
		&\lesssim \nu^{-2} \int_{0}^{t} F ^2_N(\|u^n_s\|_{H^{1-\delta}}) \|u_s^n\|_{H^{\frac{1}{2}}}^2 \|\nabla(u_s^n-\bar{u}_s)\|_{H^{-\delta}}^2 \, ds\\
		&\lesssim N^2 \nu^{-2}  \int_{0}^{t}\|u_s^n-\bar{u}_s\|_{H^{1-\delta}}^2 \, ds,
		\endaligned $$
		where in the last step we magnified $\|u^n_s\|_{H^\frac{1}{2}}$ into $\|u^n_s\|_{H^{1-\delta}}$ for $\delta \in (0,\frac{1}{4})$ and utilized the definition of the cut-off function.
		
		For the second term, we need to apply Lemma \ref{CH1 bounded for baru}. Noticing that $\sup_{t\in [0,T]} \|\bar{u}_t\|_{H^1} \lesssim \|u_0\|_{H^1}$ and $F_N(\|u_t^n\|_{H^{1-\delta}})\leq 1$, we have
		\begin{equation}\label{baru H1 norm}
			\begin{split}
				\int_{0}^{t}A_{1,2,n}(s) \, ds & \lesssim  \Big(1+\frac{3\nu}{5}\Big)^{-2} \int_{0}^{t}  \Big\| F_N(\|u_s^n\|_{H^{1-\delta}}) \Pi \big( (u^n_s-\bar{u}_s)\cdot \nabla \bar{u}_s \big) \Big\|_{H^{-1-\delta}}^2 \, ds\\
				&\lesssim  \nu^{-2} \int_{0}^{t} \|u_s^n-\bar{u}_s\|_{H^{1-\delta}}^2 \|\nabla \bar{u}_s\|_{H^{-\frac{1}{2}}}^2 \, ds\\
				&\lesssim   \nu^{-2}  \|u_0\|^2_{H^1} \int_{0}^{t} \|u_s^n-\bar{u}_s\|_{H^{1-\delta}}^2  \, ds.
			\end{split}
		\end{equation}
		As for the term $A_{1,3,n}(t)$, we use Lemmas \ref{cut-off function estimate} and \ref{CH1 bounded for baru} to get,
		$$\aligned
		\int_{0}^{t} A_{1,3,n}(s) \, ds & \lesssim \Big(1+\frac{3\nu}{5}\Big)^{-2} \int_{0}^t \Big\| \big(F_N(\|u_s^n\|_{H^{1-\delta}})-F_N(\|\bar{u}_s\|_{H^{1-\delta}})\big) \Pi \big( \bar{u}_s \cdot \nabla \bar{u}_s \big)  \Big\|^2_{H^{-1-\delta}} \, ds\\
		&\lesssim \nu^{-2} N^{-2} \int_{0}^{t} F^2_N(\|u_s^n\|_{H^{1-\delta}}) F^2_N(\|\bar{u}_s\|_{H^{1-\delta}}) \|u_s^n-\bar{u}_s\|_{H^{1-\delta}}^2 \|\bar{u}_s\|_{H^{\frac{1}{2}}}^2 \|\nabla\bar{u}_s\|_{H^{-\delta}}^2 \, ds\\
		&\lesssim \nu^{-2} \|u_0\|_{H^1}^2 \int_{0}^{t} \|u_s^n-\bar{u}_s\|_{H^{1-\delta}}^2  \, ds,
		\endaligned $$
		where the last step is due to the facts  $F_N(\|\bar{u}_s\|_{H^{1-\delta}}) \|\bar{u}_s\|_{H^{1-\delta}}  \leq N$ and $F_N(\|u_s^n\|_{H^{1-\delta}}) \leq 1$. Summarizing the above estimates for $A_{1,i,n}$, $i=1,2,3$, we arrive at
		\begin{equation}\label{A1R}
			\int_{0}^{t} A_{1,n}(s) \, ds \lesssim  \nu^{-2} (1+N^2) \big(1+\|u_0\|_{H^1}^2\big) \int_{0}^{t} \|u_s^n-\bar{u}_s\|_{H^{1-\delta}}^2  \, ds.
		\end{equation}
		
		Then we turn to the term $A_{2,n}$. Applying Theorem \ref{S limit thm} with $b=1$ and $\alpha=\delta$, it holds
		\begin{equation}\nonumber
			\begin{split}
				\E \bigg[\int_{0}^{t} A_{2,n}(s) \, ds\bigg]  &\lesssim \Big(1+\frac{3\nu}{5}\Big)^{-2}\, \E \int_{0}^{t} \Big\|  S_{\theta^n}(u_s^n)-\frac{3\nu}{5} \Delta u_s^n   \bigg\|_{H^{-1-\delta}}^2 \, ds \\
				&\lesssim   \nu^2 n^{-2\delta}  \Big(1+\frac{3\nu}{5}\Big)^{-2} \, \E \int_{0}^{t}\|u_s^n\|_{H^1}^2 \, ds\\
				&\lesssim  n^{-2\delta}  \|u_0\|_{L^2}^2.
			\end{split}
		\end{equation}
		For the last term $\big\|Z_t^n\big\|^2_{H^{1-\delta}}$, we need to apply the method of interpolation. First, by the definition of $Z_t^n$ and Lemma \ref{semigroup prop 2}, for $\tau \in (0,1)$, it holds
		$$\aligned
		\E \, \big\|Z_s^n\big\|_{H^{\tau}}^2&=\frac{3\nu}{2} \, \E \, \bigg\|\sum_{k,i} \theta_k^n \int_{0}^{s} P_{s-z} \Pi \big(\sigma_{k,i} \cdot \nabla u_z^n\big) \, dW^{k,i}_z \bigg\|_{H^\tau}^2\\
		&\lesssim \nu \, \E \bigg[\sum_{k,i} \big(\theta_k^n\big)^2 \int_{0}^{s} \big\|P_{s-z} \Pi\big(\sigma_{k,i} \cdot \nabla u_z^n\big)\big\|_{H^\tau}^2 \, dz\bigg]\\
		&\lesssim \nu \, \E \bigg[\sum_{k,i} \big(\theta_k^n\big)^2 \int_{0}^{s} \frac{1}{(1+\frac{3\nu}{5})^\tau (s-z)^\tau} \big\|\sigma_{k,i} \cdot \nabla u_z^n\big\|_{L^2}^2 \, dz \bigg]\\
		&\lesssim \nu^{1-\tau} \|\theta^n\|_{\ell^2}^2 \, \E \int_{0}^{s} \frac{1}{(s-z)^\tau}  \big\| \nabla u_z^n\big\|_{L^2}^2 \, dz,
		\endaligned $$
		and therefore by energy estimate and the assumption that $\|\theta^n\|_{\ell^2}=1$, we have
		\begin{equation}\label{stochastic convolution L2Ha}
			\begin{split}
				\int_{0}^{t} \E \, \big\|Z_s^n\big\|_{H^{\tau}}^2 \, ds &\lesssim \nu^{1-\tau}  \int_{0}^{t} \E \int_{0}^{s} \frac{1}{(s-z)^\tau}  \big\| \nabla u_z^n\big\|_{L^2}^2 \, dz ds\\
				&=\nu^{1-\tau}   \, \E \int_{0}^{t} \big\|\nabla u_z^n\big\|_{L^2}^2 \int_{z}^{t} \frac{1}{(s-z)^\tau} \, ds dz\\
				&\lesssim \frac{ \nu^{1-\tau}  T^{1-\tau}}{1-\tau}   \|u_0\|_{L^2}^2.
			\end{split}
		\end{equation}
		Besides, by \cite[Lemma 2.5]{FLD quantitative}, for $\varepsilon \in (0,1/2)$ and $p \in [1,\infty)$, it holds
		$$\E \, \bigg[\sup_{s\in [0,T]} \|Z_s^n\|_{H^{-\frac{3}{2}-\varepsilon}}^p\bigg]^{1/p} \lesssim \nu^\frac{1}{2} \Big(1+\frac{3\nu}{5}\Big)^{\frac{\varepsilon-1}{2}} \|\theta^n\|_{\ell^\infty} \|u_0\|_{L^2}\lesssim \nu^{\frac{\varepsilon}{2}}  \|\theta^n\|_{\ell^\infty} \|u_0\|_{L^2}.$$
		For any fixed $\delta \in (0,1/4)$,  choosing $\varepsilon=\delta/2 \in (0,1/8)$ in the above estimate and letting $\tau=1-\delta/2$ in \eqref{stochastic convolution L2Ha}, we use interpolation inequality, H\"older's inequality and the above two estimates to get
		\begin{equation}\label{stochastic convolution estimate}
			\begin{split}
				\E \bigg[\int_{0}^{t} \big\|Z_s^n\big\|^2_{H^{1-\delta}} \, ds\bigg]& \lesssim \int_{0}^{t} \E \,  \big\|Z_s^n\big\|_{H^{1-\frac{\delta}{2}}}^{\frac{10-2\delta}{5}} \, \big\|Z_s^n\big\|_{H^{-\frac{3}{2}-\frac{\delta}{2}}}^{\frac{2\delta}{5}}\, ds\\
				&\lesssim \int_{0}^{t} \Big(\E \,  \big\|Z_s^n\big\|_{H^{1-\frac{\delta}{2}}}^{2}\Big)^\frac{5-\delta}{5} \Big(\E \, \big\|Z_s^n\big\|_{H^{-\frac{3}{2}-\frac{\delta}{2}}}^2\Big)^{\frac{1}{2}\cdot\frac{2\delta}{5}} \, ds\\
				&\lesssim \nu^\frac{\delta}{2} \delta^{\frac{\delta-5}{5}} T^{\frac{7\delta-\delta^2}{10}} \|\theta^n\|_{\ell^\infty}^{\frac{2\delta}{5}} \|u_0\|_{L^2}^2 .
			\end{split}
		\end{equation}
		
		Summarizing the above estimates, we take integral with respect to time and then expectation on the both sides of \eqref{H-lambda norm for difference} to arrive at
		\begin{equation}\label{quantitative convergence}
			\begin{split}
				\E \bigg[\int_{0}^{t} \|u_s^n-\bar{u}_s\|_{H^{1-\delta}}^2 \, ds\bigg]
				& \lesssim   \nu^{-2} (1+N^2) \big(1+\|u_0\|_{H^1}^2\big)  \, \E \bigg[\int_{0}^{t} \|u_s^n-\bar{u}_s\|_{H^{1-\delta}}^2  \, ds\bigg]\\
				&\quad +n^{-2\delta}   \|u_0\|_{L^2}^2 +\nu^\frac{\delta}{2} \delta^{\frac{\delta-5}{5}} T^{\frac{7\delta-\delta^2}{10}} \|\theta^n\|_{\ell^\infty}^{\frac{2\delta}{5}} \|u_0\|_{L^2}^2 .
			\end{split}
		\end{equation}
		Recall that we have assumed that $ \nu^{-2} (1+N^2) \big(1+\|u_0\|_{H^1}^2\big) \ll 1$, hence it holds
		$$ 1- \nu^{-2} (1+N^2) \big(1+\|u_0\|_{H^1}^2\big)  > \nu^{-2} (1+N^2) \big(1+\|u_0\|_{H^1}^2\big)>\nu^{-2} N^2\|u_0\|_{H^1}^2. $$
		Rearrange formula \eqref{quantitative convergence} and we get
		$$\aligned
		\E \bigg[\int_{0}^{t} \|u_s^n-\bar{u}_s\|_{H^{1-\delta}}^2 \, ds\bigg]
		& \lesssim \frac{  n^{-2\delta}  \|u_0\|_{L^2}^2}{ \nu^{-2} N^2\|u_0\|_{H^1}^2 }+\frac{\nu^\frac{\delta}{2} \delta^{\frac{\delta-5}{5}} T^{\frac{7\delta-\delta^2}{10}} \|\theta^n\|_{\ell^\infty}^{\frac{2\delta}{5}} \|u_0\|_{L^2}^2  }{ \nu^{-2} N^2\|u_0\|_{H^1}^2 }\\
		&\lesssim \nu^2 n^{-2\delta} N^{-2} +\nu^{\frac{\delta}{2}+2} \delta^{\frac{\delta-5}{5}} T^{\frac{7\delta-\delta^2}{10}} N^{-2} \|\theta^n\|_{\ell^\infty}^{\frac{2\delta}{5}}.
		\endaligned $$
		Taking the supremum for the time $t\in [0,T]$ and noticing that the estimate on the right-hand side is independent of $t$, we complete the proof.
	\end{proof}
	
	\begin{remark}\label{rmk on quant conv rate}
		1. Due to the existence of the cut-off function $F_N(\|u\|_{H^{1-\delta}})$, it is necessary to take $H^{1-\delta}$-norm on the both sides of formula \eqref{H-lambda norm for difference}. However, the stochastic convolution $Z^n$ is small enough only under some negative Sobolev norm. To address this, we employ the method of interpolation, together with the fact that $\{Z^n \}_n$ is bounded in $L^2([0,T];H^\tau)$ for any $\tau \in (0,1)$, to get a sufficiently small estimate for $\|Z^n\|_{L^2([0,T];H^{1-\delta})}$, see \eqref{stochastic convolution estimate} above for details.
		
		2. In formula \eqref{stochastic convolution estimate}, we choose $\varepsilon=\delta/2$ and $\tau=1-\delta/2$ for simplicity. In fact, one can take other values for these two parameters and then apply interpolation to get a slightly different estimate, which also yields a sufficiently small bound on $\E \|u-\bar{u}\|^2_{L^2 ([0,T];H^{1-\delta})}$  under some suitable assumptions.

		3. For hyperviscous stochastic GMNSE, we can utilize Theorem \ref{scaling limit} to write the mild solutions of the sequence to equations \eqref{sequence of thetaR} with $\Lambda \in (1,2)$ as
		$$\aligned
		u^n_t&=\bar{P}_t u_0-\int_{0}^{t} \bar{P}_{t-s} \Big[F_N(\|u^n_s\|_{H^{1-\delta}}) \Pi(u^n_s\cdot \nabla u^n_s) \Big]\, ds+\int_{0}^{t} \bar{P}_{t-s} \Big(S_{\theta^n}(u_s^n)-\frac{3\nu}{5} \Delta u_s^n \Big) \, ds\\
		&\quad-\int_{0}^{t} \bar{P}_{t-s} (-\Delta)^\Lambda u_s^n \, ds+ \bar{Z}_t^n, \quad \forall n\geq 1,
		\endaligned $$
		where $\bar{P}_t=e^{\frac{3\nu}{5}t\Delta}$ and
		$$\bar{Z}_t^n:=\sqrt{\frac{3\nu}{2}}\sum_{k,i}\theta_k^n \int_{0}^{t} \bar{P}_{t-s}  \Pi\big(\sigma_{k,i} \cdot \nabla u_s^n\big) \, dW_s^{k,i}.$$
		Similarly, the mild solutions to \eqref{limit eq} can be written as
		$$\bar{u}_t=\bar{P}_t u_0-\int_{0}^{t} \bar{P}_{t-s} \Big(F_N(\|\bar{u}_s\|_{H^{1-\delta}}) \Pi(\bar{u}_s\cdot \nabla \bar{u}_s) \Big)\, ds-\int_{0}^{t}\bar{P}_{t-s} (-\Delta)^\Lambda \bar{u}_s \, ds. $$
		Making difference between $u^n_t$ and $\bar{u}_t$, we take $H^{1-\delta}$-norm and then square to obtain
		\begin{equation}\label{difference for hyperviscous}
			\begin{split}
				\|u_t^n-\bar{u}_t\|_{H^{1-\delta}}^2 &\lesssim \bigg\|\! \int_{0}^{t}\! \bar{P}_{t-s} \Big(F_N(\|u^n_s\|_{H^{1-\delta}}) \Pi(u^n_s\cdot \nabla u^n_s)-F_N(\|\bar{u}_s\|_{H^{1-\delta}}) \Pi(\bar{u}_s\cdot \nabla \bar{u}_s) \Big) ds \bigg\|_{H^{1-\delta}}^2 \\
				&\quad +\bigg\|\int_{0}^{t} \bar{P}_{t-s}\Big(S_{\theta^n}(u_s^n)-\frac{3\nu}{5} \Delta u_s^n \Big) \, ds \bigg\|_{H^{1-\delta}}^2+ \big\|\bar{Z}_t^n\|_{H^{1-\delta}}^2 \\
&\quad + \bigg\|\int_{0}^{t} \bar{P}_{t-s} \big[(-\Delta)^\Lambda (u_s^n-\bar{u}_s) \big]\, ds \bigg\|_{H^{1-\delta}}^2.
			\end{split}
		\end{equation}
		The estimates of these terms are similar to those of stochastic GMNSE except for the last term. In fact, using Lemma \ref{semigroup property}, we have
		$$\aligned
		\int_{0}^{t} \bigg\|\int_{0}^{s} \bar{P}_{s-z}\big[(-\Delta)^\Lambda (u_z^n-\bar{u}_z) \big] \, dz\bigg\|_{H^{1-\delta}}^2 \, ds &\lesssim \nu^{-2} \int_{0}^{t} \big\|(-\Delta)^\Lambda (u_s^n-\bar{u}_s)\big\|^2_{H^{-1-\delta}} \, ds \\
		&\lesssim \nu^{-2} \int_{0}^{t}\|u_s^n-\bar{u}_s\|^2_{H^{2\Lambda-1-\delta}} \, ds.
		\endaligned $$
		As $2\Lambda-1-\delta>1-\delta$, we cannot magnify it into $\|u^n-\bar{u}\|_{L^2([0,T];H^{1-\delta})}$ and then move it to the left-hand side of the equation obtained after time integral of \eqref{difference for hyperviscous}.  However, recalling that $\bar{u}, \, u^n \in L^2([0,T];H^\Lambda)$, we can make a restriction that $2\Lambda-1-\delta \leq \Lambda$, i.e., $\Lambda \leq 1+\delta$ to bound this term. Then we can apply arguments similar to those used for the stochastic GMNSE to obtain the quantitative convergence rate.
	\end{remark}

	\section{Large deviation principle} \label{sec-LDP}
	In this section, we will apply weak convergence method to prove our final result on LDP. We first introduce this approach in Section \ref{subsec-weak convergence method}, then we prove Theorem \ref{LDP thm} in Section \ref{subsec-prf of LDP}.

	\subsection{Weak convergence method}\label{subsec-weak convergence method}
	To begin with, we introduce some spaces. Let $(\Omega,\mathcal{F},(\mathcal{F})_t,\P)$ be a stochastic basis satisfying the usual conditions. Let $U$ be a Hilbert space, on which we have a trace class operator $Q$. Denote $U_0=Q^{\frac{1}{2}} (U)$ as Cameron-Martin space, endowed with the inner product
	$$\<g,h\>_{U_0}=\big\< Q^{-\frac{1}{2}} g, Q^{-\frac{1}{2}} h\big\>_U, \quad g,h \in U_0.$$
	Then we define the space
	\begin{equation}\label{SM space}
		S^M=S^M(U_0):=\bigg\{ g\in L^2([0,T];U_0): \int_{0}^{T} \|g(s)\|_{U_0}^2 \, ds \leq M\bigg\}.
	\end{equation}
	If $S^M$ is endowed with the weak topology, it is a Polish space. Denote $\mathcal{P}_2(U_0)$ as the class of $U_0$-valued $\mathcal{F}_t$-predictable processes $g$, such that $$\int_{0}^{T} \|g(s)\|_{U_0}^2 \, ds<\infty, \quad \P \text{-a.s.} ,$$
	and define
	$$\mathcal{P}_2^M=\mathcal{P}_2^M(U_0):=\big\{g\in \mathcal{P}_2(U_0):g(\cdot, \omega)\in S^M(U_0) \ \text{for} \ \P\text{-a.s.} \ \omega \big\}.$$
	In the sequel, we always denote $\mathcal{E}$ and $\mathcal{E}_0$ as Polish spaces.
	
	Then we provide the definitions of rate function and LDP.
	\begin{definition}\label{def of rate function}
		A function $I: \mathcal{E} \rightarrow [0,\infty]$ is called a rate function if for any $M<\infty$, the level set $\{f\in \mathcal{E}:I(f)\leq M\}$ is a compact subset of $\mathcal{E}$. A family of rate functions $I_x$ on $\mathcal{E}$, indexed by $x\in \mathcal{E}_0$, is said to have compact level sets on compacts if for all compact subsets $\mathcal{K}$ of $\mathcal{E}_0$ and each $M<\infty$, $\cup_{x\in \mathcal{K}} \{f\in \mathcal{E}: I_x(f)\leq M\}$ is a compact subset of $\mathcal{E}$.
	\end{definition}
	
	\begin{definition}\label{def of LDP}
		Let $I$ be a rate function on $\mathcal{E}$ and denote $I(A):=\inf_{f\in A} I(f)$ for any Borel set $A \in \mathcal{B}(\mathcal{E})$. We say that a family of $\mathcal{E}$-valued random variables $\{X^\varepsilon\}$  satisfies LDP with speed $\varepsilon$ and rate function $I$ if
		$$-I(A^\circ) \leq \liminf_{\varepsilon \rightarrow 0} \varepsilon \log \P(X^\varepsilon \in A) \leq \limsup_{\varepsilon \rightarrow 0} \varepsilon \log \P(X^\varepsilon \in A) \leq -I(\bar{A}),$$
		where $A^\circ$ and $\bar{A}$ are the interior and closure of $A$, respectively.
	\end{definition}
	
	We can see from Definition \ref{def of LDP} that, the theory of LDP is concerned with events $A$ for which the probabilities $ \P(X^\varepsilon \in A)$ decay exponentially fast as $\varepsilon \rightarrow 0$. This decay rate is characterized by the rate function defined in Definition \ref{def of rate function}. Usually, $X^\varepsilon$ are probabilistically strong solutions to some SPDEs driven by $\sqrt{\varepsilon} W$ with initial value $x \in \mathcal{E}_0$. And there exist measurable maps $\mathcal{G}^\varepsilon: \mathcal{E}_0 \times C_t^0 U \rightarrow \mathcal{E}$, such that the solutions can be written as $X^{\varepsilon,x}=\mathcal{G}^\varepsilon(x,\sqrt{\varepsilon}W)$.
	
	Recall the notation $\text{Int}(g)(\cdot)=\int_{0}^{\cdot} g(s) \, ds$ introduced above Theorem \ref{LDP thm}.
	To apply the weak convergence method to obtain the LDP result,  the following two assumptions need to be verified (see \cite[Theorem 5]{BDM08} for more details):
	\begin{hypothesis}\label{LDP hypothesis}
		There exists a measurable map $\mathcal{G}^0: \mathcal{E}_0 \times C_t^0 U \rightarrow \mathcal{E}$, such that
		
		1. For any $M<\infty$ and compact set $\mathcal{K} \subset \mathcal{E}_0$, $\Gamma_{\mathcal{K},M}:=\big\{ \mathcal{G}^0(x, \text{Int}(g)): g\in S^M, x\in \mathcal{K}\big\}$ is a compact subset of $\mathcal{E}$;
		
		2. Consider $M<\infty$ and families $\{x^\varepsilon\} \subset \mathcal{E}_0$, $\{g^\varepsilon\} \subset \mathcal{P}_2^M$ satisfying: $x^\varepsilon \rightarrow x$ and $g^\varepsilon$ converge in law to $g$ as $S^M$-valued random elements as $\varepsilon \rightarrow 0$. Then $\mathcal{G}^\varepsilon\big(x^\varepsilon, \sqrt{\varepsilon} W+\text{Int}(g^\varepsilon)\big)$ converges in law to $\mathcal{G}^0\big(x,\text{Int}(g)\big)$ in the topology of $\mathcal{E}$.
		
	\end{hypothesis}
	
	\subsection{Proof of Theorem \ref{LDP thm}} \label{subsec-prf of LDP}
	We will provide the verifications of the above two hypotheses in Lemmas \ref{hypothesis 1} and \ref{hypothesis 2}, respectively.	To begin with, we give the definition of weak solutions to the skeleton equation \eqref{skeleton eq}. Recall that the Sobolev spaces $H^s\ (s\in \R)$ consist of divergence-free vector fields.
	
	\begin{definition}\label{def of solution of skeleton eq}
		Let $u_0 \in L^2$ and $g\in L^2([0,T];H^r)$ for $r \in (0,\frac{3}{2})$ be divergence-free. We say that  $u$ is a weak solution to skeleton equation \eqref{skeleton eq} if for any $\phi \in C^\infty(\T^3,\R^3)$ satisfying ${\rm div} \phi=0$, it holds
		$$\aligned
		\<u(t), \phi\>&=\<u_0,\phi\>+\int_{0}^{t} F_N(\|u(s)\|_{H^{1}}) \big\<u(s) \cdot \nabla \phi, u(s) \big\> \, ds-\int_{0}^{t} \big\<u(s),(-\Delta)^\Lambda \phi \big\>\, ds\\
		&\quad+\frac{3\nu}{5}\int_{0}^{t} \<u(s), \Delta \phi\> \, ds-\int_{0}^{t} \big\<u(s) \cdot \nabla \phi, g(s) \big\> \, ds.
		\endaligned $$
	\end{definition}
	
	Based on the above definition, we claim
	\begin{lemma}\label{well posedness for skeleton eq}
		Given arbitrary $T>0$, $\Lambda\in (1,2)$ and $r \in (0,\frac{3}{2})$ satisfying $\Lambda+r>\frac{5}{2}$, for any initial value $u_0 \in L^2$ and $g\in L^2([0,T];H^r)$, there exists a unique weak solution $u \in L^2([0,T];H^\Lambda) \cap L^\infty([0,T];L^2)$ to skeleton equation \eqref{skeleton eq}.
	\end{lemma}

	\begin{proof}
		We will divide the proof into two steps, where we respectively prove the uniqueness and existence of the weak solution.
		
		\textbf{Step 1. Uniqueness.} Suppose $u_1$ and $u_2$ are two weak solutions of \eqref{skeleton eq} with the same initial value $u_0$ and $g$, then set $\xi:=u_1-u_2$. To apply the energy equality, we need to verify the assumption in \cite[Theorem 2.12]{RL 1990}, i.e.,
		\begin{equation}\label{Lions-Magenes}
			\int_{0}^{T} \|\xi (t)\|_{H^\Lambda}^2+\|\partial_t \xi\|_{H^{-\Lambda}}^2 \, dt <+\infty.
		\end{equation}
		The first term is trivial by the regularity of $\xi$. For the second term, we only discuss the term associated to $g$; as for the verification of the nonlinear term, it is similar to the arguments below \eqref{ito verification}. Notice that
		$$\big\|\Pi(g\cdot \nabla \xi) \big\|_{H^{-\Lambda}} \lesssim \|g\xi\|_{H^{1-\Lambda}} \lesssim \|g\|_{H^r}\|\xi\|_{H^{\frac{5}{2}-\Lambda-r}},$$
		hence for $\Lambda+r>\frac{5}{2}$ and $g\in L^2([0,T];H^r)$, we have
		$$\int_{0}^{T} \big\|\Pi(g\cdot \nabla \xi) \big\|_{H^{-\Lambda}}^2 \, dt \lesssim \|u_0\|^2_{L^2} \int_{0}^{T}\|g\|_{H^r}^2  \, dt<+\infty.$$
		Now we can make energy estimate. Integrating by parts, it holds
		\begin{equation}\label{energy estimate for skeleton eq}
			\begin{split}
				\frac{1}{2} \frac{d}{dt} \|\xi\|_{L^2}^2
				&=-\big\<\xi, F_N(\|u_1\|_{H^{1}}) \Pi(u_1 \cdot \nabla u_1)- F_N(\|u_2\|_{H^{1}}) \Pi(u_2 \cdot \nabla u_2) \big\> -\big\|(-\Delta)^{\frac{\Lambda}{2}}\xi \big\|_{L^2}^2\\
				&\quad-\frac{3\nu}{5}\|\nabla \xi\|_{L^2}^2+\big\<\xi, \Pi(g\cdot \nabla \xi)\big\>.
			\end{split}
		\end{equation}
		The last term vanishes as $g$ and $\xi$ are both divergence-free. As for the first term, we can use the estimates above formula \eqref{nonlinear difference estimate} with $\delta=0$ and then apply Young's inequality to get
		\begin{equation}\label{nonlinear estimate}
			\big|\big\<\xi, F_N(\|u_1\|_{H^{1}}) \Pi(u_1 \cdot \nabla u_1)- F_N(\|u_2\|_{H^{1}}) \Pi(u_2 \cdot \nabla u_2) \big\>\big|	\lesssim  \frac{3\nu}{5} \|\xi\|_{H^1}^2+C(N,\nu) \|\xi\|_{L^2}^2 ,
		\end{equation}
		where $C(N,\nu)>0$ is a finite constant. Inserting this estimate into \eqref{energy estimate for skeleton eq}  yields
		$$\aligned
		\frac{1}{2} \frac{d}{dt} \|\xi\|_{L^2}^2 +\|\xi\|_{H^{\Lambda}}^2\lesssim C(N,\nu) \|\xi\|_{L^2}^2.
		\endaligned $$
		Hence the uniqueness follows if we apply Gr\"onwall's lemma for the above inequality.
		
		\textbf{Step 2. Existence.} Consider Galerkin approximations for equation \eqref{skeleton eq}:
		\begin{equation}\label{ske galerkin}
			\partial_t u_m+F_N(\|u_m\|_{H^{1}}) \Pi_m (u_m \cdot \nabla u_m)=-(-\Delta)^\Lambda u_m+\frac{3\nu}{5}\Delta u_m+\Pi_m(g\cdot\nabla u_m), \ u_m(0)=\Pi_m u_0,
		\end{equation}
		where $\Pi_m$ is the projection defined below \eqref{eq-LDP}. Making energy estimate for the above equation and noticing that $g$ and $u$ are both divergence-free, we get
		$$\frac{1}{2} \frac{d}{dt} \|u_m\|_{L^2}^2=-\big\|(-\Delta)^{\frac{\Lambda}{2}}u_m\big\|_{L^2}^2-\frac{3\nu}{5}\|\nabla u_m\|_{L^2}^2,$$
		which implies that for all $t\geq 0$,
		$$\|u_m(t)\|_{L^2}^2+2\int_{0}^{t} \|u_m(s)\|_{H^\Lambda}^2 \, ds+\frac{6\nu}{5}\int_{0}^{t} \|u_m(s)\|_{H^1}^2 \, ds = \|\Pi_m u_0\|_{L^2}^2 \leq \|u_0\|_{L^2}^2.$$
		Hence $\{u_m\}_{m\geq 1} \subset L^\infty([0,T];L^2) \cap L^2([0,T];H^\Lambda)$ for $\Lambda \in (1,2)$.
		
		Next, we want to prove that $\{u_m\}_m \subset W^{\alpha,2}([0,T];H^{-\beta})$ for $\alpha \in (0,\frac{1}{2})$ and $\beta>\frac{5}{2}$, which is a direct corollary of
		$$\big\|u_m(t)-u_m(s)\big\|_{H^{-\beta}} \lesssim |t-s|^{\frac{1}{2}}, \quad 0\leq s<t \leq T.$$
		As the arguments are similar to those in the proof of Lemma \ref{bdd in sobolev space}, we only give the estimate for the term associated to $g$. Since $\beta>\frac{5}{2}$ and $r \in (0,\frac{3}{2})$, for any $z \in [s,t]$, we have
		$$\big\|\Pi_m(g(z) \cdot \nabla u_m(z))\big\|_{H^{-\beta}} \leq \|g(z) \cdot \nabla u_m(z)\|_{H^{r-\frac{5}{2}}}  \lesssim \|g(z) u_m(z)\|_{H^{r-\frac{3}{2}}} \lesssim \|g(z)\|_{H^r} \|u_m(z)\|_{L^2},$$
		and therefore for $g\in L^2([0,T];H^r)$, it holds
		$$	\int_{s}^{t} \big\|\Pi_m(g(z) \cdot \nabla u_m(z))\big\|_{H^{-\beta}} \, dz
		\lesssim \|u_0\|_{L^2} \int_{s}^{t} \|g(z)\|_{H^r}  \, dz \leq \|u_0\|_{L^2}  \|g\|_{L^2([0,T];H^r)} |t-s|^{\frac{1}{2}}.$$
		Hence by the definition of fractional Sobolev norm, we have
		$$\aligned
		\|u_m\|^2_{W^{\alpha,2}([0,T];H^{-\beta})} &= \int_{0}^{T} \int_{0}^{T} \frac{\|u_m(t)-u_m(s)\|_{H^{-\beta}}^2}{|t-s|^{1+2\alpha}} \, dtds + \|u_m\|^2_{L^2([0,T];H^{-\beta})}\\
		&\lesssim \int_{0}^{T} \int_{0}^{T} \frac{1}{|t-s|^{2\alpha}} \, dtds+T \|u_0\|^2_{L^2} <+\infty.
		\endaligned $$
		Similarly, we can prove that $\{u_m\}_m \subset W^{\frac{1}{3},4}([0,T];H^{-\beta})$ for $\beta>\frac{5}{2}$.
		Applying Theorem \ref{compact embedding} (i') and (ii), we can find a subsequence $\{m_j\}_{j\geq 1}$, such that $u_{m_j}$ converges strongly to some limit $u$ in  $L^2([0,T];H^{1}) \cap C([0,T];H^{-\gamma})$ for any $\gamma \in (0,\frac{1}{2})$.
		
		Now we aim to verify that $u$ is the weak solution to \eqref{skeleton eq} in the sense of Definition \ref{def of solution of skeleton eq}. Notice that $u_{m_j}$ satisfies the following identity for any divergence-free vector field $\phi \in C^\infty$:
		$$\aligned
		\<u_{m_j}(t), \phi\>&=\<\Pi_{m_j}u_0,\phi\>+\int_{0}^{t} F_N(\|u_{m_j}(s)\|_{H^{1}}) \big\<u_{m_j}(s) \cdot \nabla \Pi_{m_j}\phi, u_{m_j}(s) \big\> \, ds\\
		& -\int_{0}^{t}\! \<u_{m_j}(s), (-\Delta)^\Lambda \phi \> \, ds +\frac{3\nu}{5}\int_{0}^{t} \! \<u_{m_j}(s), \Delta \phi\> \, ds-\int_{0}^{t}\! \big\<u_{m_j}(s) \cdot \nabla \Pi_{m_j}\phi, g(s)\big\> \, ds.
		\endaligned $$
		In fact, the convergence for the first, the third and the fourth terms on the right-hand side are trivial as $\phi$ is a smooth vector field. For the nonlinear term, the proof follows a similar approach to \eqref{nonlinear convergence J1J2} and the subsequent arguments. We mention that, for the term $J_2$ defined in \eqref{nonlinear convergence J1J2}, the estimate in \eqref{J2 estimate} requires a slight modification and now takes the form:
		$$\aligned
		\big|F_N(\|u_{m_j}\|_{H^{1}})-F_N(\|u\|_{H^{1}})\big| \lesssim \|u_{m_j}-u\|_{H^{1}}.
		\endaligned $$
		Using the fact that $u_{m_j}$ converges strongly to $u$ in $L^2([0,T];H^{1})$, $J_2$ vanishes as $j \rightarrow \infty$, and therefore we get the convergence for the nonlinear term.
		
		Finally we prove the convergence for the term associated to $g$. By the triangle inequality, for any $t\in [0,T]$, it holds (we omit the time parameter $s$ for convenience)
		$$\aligned
		&\quad \ \bigg|\int_{0}^{t} \<u_{m_j} \cdot \nabla \Pi_{m_j}\phi, g\> \, ds-\int_{0}^{t} \<u \cdot \nabla \phi, g\> \, ds\bigg|\\
		&\leq \int_{0}^{t}  \big|\big\<(u_{m_j}-u) \cdot \nabla \Pi_{m_j}\phi, g \big\>\big| \, ds+ \int_{0}^{t} \big|\big\<u \cdot \nabla (\Pi_{m_j} \phi-\phi),g \big\>\big| \, ds.
		\endaligned $$
		We will show that these two terms vanish as $j\rightarrow \infty$.  For $\gamma \in (0,\frac{1}{2})$, we have
		$$\aligned
		\int_{0}^{t}  \big|\big\<(u_{m_j}-u) \cdot \nabla \Pi_{m_j}\phi, g\big\>\big| \, ds &\leq \int_{0}^{t} \big\|(u_{m_j}-u) \cdot \nabla \Pi_{m_j}\phi \big\|_{H^{-1}} \|g\|_{H^1} \, ds\\
		&\leq \int_{0}^{t}  \|u_{m_j}-u\|_{H^{-\gamma}} \|\nabla \Pi_{m_j}\phi\|_{H^{\gamma+\frac{1}{2}}} \|g\|_{H^1} \, ds\\
		&\leq T^{\frac{1}{2}}  \|\phi\|_{H^2} \|u_{m_j}-u\|_{C([0,T];H^{-\gamma})} \|g\|_{L^2([0,T];H^1)}.
		\endaligned $$
		Since $g\in L^2([0,T];H^r)$ for any $r\in (0,\frac{3}{2})$, $\|g\|_{L^2([0,T];H^1)}$ is finite. By the condition that $u_{m_j}$ converges strongly to $u$ in $C([0,T];H^{-\gamma})$, this term tends to zero as $j \rightarrow \infty$.
		For the other term, applying H\"older's inequality and Sobolev embedding theorem, we have
		$$\aligned
		\int_{0}^{t} \big|\big\<u \cdot \nabla (\Pi_{m_j} \phi-\phi),g\big\>\big| \, ds &\leq \int_{0}^{t} \|u\|_{L^2} \big\|\nabla(\Pi_{m_j} \phi-\phi)\big\|_{L^4} \|g\|_{L^4} \, ds\\
		&\leq \int_{0}^{t} \|u\|_{L^2} \big\|\nabla(\Pi_{m_j} \phi-\phi)\big\|_{H^{\frac{3}{4}}} \|g\|_{H^{\frac{3}{4}}} \, ds\\
		&\lesssim T^{\frac{1}{2}}\|u_0\|_{L^2} \|\Pi_{m_j}\phi-\phi\|_{H^\frac{7}{4}} \|g\|_{L^2([0,T];H^{\frac{3}{4}})},
		\endaligned $$
		which also vanishes as $j\rightarrow \infty$ since $\phi$ is smooth. Summarizing the above discussions, we have proved that $u$ is a weak solution to equation \eqref{skeleton eq} and the proof of Lemma \ref{well posedness for skeleton eq} is done.
	\end{proof}

	Now we can state the following result which will help us to verify the first condition of Hypothesis \ref{LDP hypothesis}.
	
	\begin{lemma}\label{LDP convergence lemma}
		Suppose $u^i$ is a solution to skeleton equation \eqref{skeleton eq} associated to data $u_0^i \in L^2$ and $g^i \in L^2([0,T];H^r)$, where $i=1,2$ and $r\in (0,\frac{3}{2})$. If $\Lambda+r>\frac{5}{2}$, it holds
		$$ \|u^1-u^2\|_{L^\infty([0,T];L^2)} +\|u^1-u^2\|_{L^2([0,T];H^\Lambda)} \lesssim \|u_0^1-u_0^2\|_{L^2}+\|u^2_0\|_{L^2} \|g^1-g^2\|_{L^2([0,T];H^r)}.$$
		Besides, given sequences $u_0^n \rightharpoonup u_0$ and $g^n \rightharpoonup g$, the associated solutions $u^n$ converges weakly to $u$ in $L^2([0,T];H^\Lambda)$, as well as strongly in $C([0,T];H^{-\lambda}) \cap L^2([0,T];L^2)$ for any $\lambda>0$.
	\end{lemma}
	
	\begin{proof}
		We begin with proving the first estimate. As $u^1$, $u^2$ are both solutions to equation \eqref{skeleton eq} associated to $g^1$ and $g^2$ respectively, we have
		$$\aligned
		\partial_t(u^1-u^2)=&-\big[F_N(\|u^1\|_{H^{1}}) \Pi(u^1\cdot \nabla u^1)-F_N(\|u^2\|_{H^{1}}) \Pi(u^2\cdot \nabla u^2)\big]-(-\Delta)^\Lambda (u^1-u^2)\\
		&+\frac{3\nu}{5} \Delta(u^1-u^2)+\Pi(g^1\cdot \nabla u^1-g^2\cdot\nabla u^2).
		\endaligned $$
		Recalling that we have verified in \eqref{Lions-Magenes} that energy equality is valid, it holds
		$$\aligned
		\frac{1}{2} \frac{d}{dt} \big\|u^1-u^2\big\|_{L^2}^2&=-\big\< u^1-u^2, F_N(\|u^1\|_{H^{1}}) \Pi(u^1\cdot \nabla u^1)-F_N(\|u^2\|_{H^{1}}) \Pi(u^2\cdot \nabla u^2)\big\>\\
		&-\big\|u^1-u^2\big\|_{H^\Lambda}^2-\frac{3\nu}{5} \big\|\nabla(u^1-u^2)\big\|_{L^2}^2+\big\<u^1-u^2, \Pi(g^1\cdot \nabla u^1-g^2\cdot\nabla u^2)\big\>.
		\endaligned $$
		For the first term,  \eqref{nonlinear estimate} implies
		$$\aligned
		&\quad \big|\big\< u^1-u^2, F_N(\|u^1\|_{H^{1}}) \Pi(u^1\cdot \nabla u^1)-F_N(\|u^2\|_{H^{1}}) \Pi(u^2\cdot \nabla u^2)\big\>\big| \\
		&\lesssim \frac{3\nu}{5} \big\|u^1-u^2\big\|_{H^1}^2+C(\nu,N) \big\|u^1-u^2\big\|_{L^2}^2.
		\endaligned $$
		So we only estimate the last term concerning $g$. The triangle inequality yields
		$$\aligned
		&\ \quad\big|\big\<u^1-u^2, \Pi(g^1\cdot \nabla u^1-g^2\cdot\nabla u^2)\big\>\big|\\
		&\leq \big|\big\<u^1-u^2, \Pi(g^1 \cdot \nabla (u^1-u^2))\big\>\big|+\big|\big\<u^1-u^2, \Pi((g^1-g^2)\cdot \nabla u^2)\big\>\big|.
		\endaligned $$
		As $g^1$ and $u^i, i=1,2$ are divergence-free, the former part vanishes. Besides, for $\Lambda+r>\frac{5}{2}$,
		$$\aligned
		\big|\big\<u^1-u^2, \Pi((g^1-g^2)\cdot \nabla u^2)\big\>\big|&=\big|\big\<\nabla(u^1-u^2) , u^2 (g^1-g^2)\big\>\big|\\
		&\leq \big\|\nabla(u^1-u^2)\big\|_{H^{\Lambda-1}} \big\|u^2 (g^1-g^2)\big\|_{H^{1-\Lambda}} \\
		&\lesssim \big\|u^1-u^2\big\|_{H^{\Lambda}} \|u^2\|_{H^{\frac{5}{2}-\Lambda-r}} \big\|g^1-g^2\big\|_{H^r}\\
		&\leq \frac{1}{2} \big\|u^1-u^2\big\|_{H^\Lambda}^2+c \, \|u^2_0\|_{L^2}^2 \big\|g^1-g^2\big\|_{H^r}^2,
		\endaligned $$
		where $c>0$ is a finite constant. Summarizing the above discussions, we get
		$$\frac{1}{2} \frac{d}{dt} \big\|u^1-u^2\big\|_{L^2}^2 +\frac{1}{2}\big\|u^1-u^2\big\|_{H^\Lambda}^2 \lesssim C(\nu,N) \big\|u^1-u^2\big\|_{L^2}^2+c \, \|u^2_0\|_{L^2}^2 \big\|g^1-g^2\big\|_{H^r}^2.$$
		Applying Gr\"onwall's inequality and then integrating the above inequality with respect to time, we get the first desired estimate.
		
		Now we turn to prove the next convergence result. For brevity, we will use the notation $ C^0_t L^2_x$ to represent
		$C([0,T];L^2)$, and similar notations  will be employed for  $L^2_t H^\Lambda_x$, $C_t^\frac{1}{2} H^{-\Lambda}_x$, etc. By Lemma \ref{well posedness for skeleton eq}, for given data $u_0^n$ and $g^n$, the corresponding weak solution $u^n$ to \eqref{skeleton eq} belongs to the space $ C^0_t L^2_x \cap L^2_t H^\Lambda_x$, and we aim to prove it belongs to $ C_t^\frac{1}{2} H_x^{-\Lambda}$ as well. For any $0\leq s<t \leq T$ and $\Lambda \in (1,2)$, we integrate equation \eqref{skeleton eq} with respect to time from $s$ to $t$ and then take $H^{-\Lambda}$-norm to get
		$$\aligned
		\big\|u^n(t)-u^n(s)\big\|_{H^{-\Lambda}}  \leq &\int_{s}^{t} \big\|F_N(\|u^n(z)\|_{H^1}) \Pi\big(u^n(z)\cdot\nabla u^n(z)\big)\big\|_{H^{-\Lambda}} \, dz +\int_{s}^{t} \|u^n(z)\|_{H^{\Lambda}} \, dz \\
		&+ \frac{3\nu}{5} \int_{s}^{t} \|u^n(z)\|_{H^{-\Lambda+2}} \, dz+ \int_{s}^{t} \big\| \Pi \big(g^n(z) \cdot \nabla u^n(z)\big) \big\|_{H^{-\Lambda}} \, dz.
		\endaligned $$
		We estimate each term respectively. For the nonlinear term, by interpolation inequality and $F_N(\|u^n\|_{H^1}) \leq \frac{N}{\|u^n\|_{H^1}}$, it holds
		\begin{equation}\label{nonlinear H-1 norm}
			\big\|F_N(\|u^n\|_{H^1})  \Pi(u^n \cdot \nabla u^n)\big\|_{H^{-1}} \leq F_N(\|u^n\|_{H^1}) \|u^n\|_{H^{\frac{1}{2}}} \|\nabla u^n\|_{L^2} \leq N \|u^n\|_{H^1}^\frac{1}{2} \|u^n\|_{L^2}^\frac{1}{2};
		\end{equation}
		hence we can use H\"older's inequality to obtain
		$$\aligned
		\int_{s}^{t}\big\|F_N(\|u^n(z)\|_{H^1}) \Pi\big(u^n(z)\cdot\nabla u^n(z)\big)\big\|_{H^{-\Lambda}} \, dz \leq N \|u^n_0\|_{L^2}^\frac{1}{2} \int_{s}^{t} \|u^n(z)\|_{H^1}^\frac{1}{2} \, dz \lesssim N \|u^n_0\|_{L^2} |t-s|^\frac{3}{4}.
		\endaligned $$
		For the second and the third terms,  H\"older's inequality and energy estimate yield
		$$\int_{s}^{t} \|u^n(z)\|_{H^{\Lambda}} \, dz \lesssim \|u^n_0\|_{L^2} |t-s|^\frac{1}{2};$$
		and for $\Lambda \in (1,2)$, it holds
		$$\frac{3\nu}{5} \int_{s}^{t} \|u^n(z)\|_{H^{-\Lambda+2}} \, dz \lesssim \nu \int_{s}^{t} \|u^n(z)\|_{H^{\Lambda}} \, dz \lesssim \nu \|u^n_0\|_{L^2} |t-s|^\frac{1}{2}.$$
		For the last term corresponding to $g^n$, as $\Lambda+r>\frac{5}{2}$, it holds
		$$\big\|\Pi\big(g^n(z) \cdot \nabla u^n(z)\big) \big\|_{H^{-\Lambda}}  \lesssim \|g^n(z)u^n(z)\|_{H^{1-\Lambda}} \lesssim \|g^n(z)\|_{H^r} \|u^n(z)\|_{H^{\frac{5}{2}-\Lambda-r}} \lesssim \|g^n(z)\|_{H^r} \|u^n_0\|_{L^2},$$
		and therefore for $g^n\in L^2_t H^r_x$, we have
		$$\int_{s}^{t} \big\| \Pi \big(g^n(z) \cdot \nabla u^n(z)\big) \big\|_{H^{-\Lambda}} \, dz \lesssim \|u^n_0\|_{L^2} \int_{s}^{t}
		\|g^n(z)\|_{H^r}  \, dz \leq \|u^n_0\|_{L^2} \|g^n\|_{L^2_t H^r_x} |t-s|^\frac{1}{2}.$$
		Summarizing the above discussions, for any $0 \leq s<t \leq T$, we arrive at
		$$\big\|u^n(t)-u^n(s)\big\|_{H^{-\Lambda}} \lesssim N \|u^n_0\|_{L^2} |t-s|^\frac{3}{4}+ (1+\nu) \|u^n_0\|_{L^2} |t-s|^\frac{1}{2}+\|u^n_0\|_{L^2} \|g^n\|_{L^2_t H^r_x} |t-s|^\frac{1}{2}.$$
		Hence, given $u_0^n$ and $g^n$, the associated solutions $u^n$ satisfies uniform bounds in $C_t^\frac{1}{2} H^{-\Lambda}_x \cap C^0_t L^2_x \cap L^2_t H^\Lambda_x$ for $\Lambda \in (1,2)$. By Aubin-Lions Lemma, $\{u^n\}_n$ is sequentially compact in $C^0_t H_x^{-\lambda} \cap L^2_t L^2_x$ for any $\lambda>0$, implying that $u^n$ converges strongly to some limit $u$ in this space. Finally, we can show that the limit $u$ is a weak solution to \eqref{skeleton eq} by standard arguments.
	\end{proof}

	Thanks to Lemma \ref{well posedness for skeleton eq}, now we can claim that the solution map $\mathcal{G}^0$ defined above Theorem \ref{LDP thm} is meaningful, and therefore the unique solution $u\in L^\infty_t L^2_x \cap L^2_t H^1_x$ to skeleton equation \eqref{skeleton eq} can be written as $u=\mathcal{G}^0(u_0, \text{Int}(g))$, where $u_0 \in L^2$ and $g\in L^2_t H_x^r$. Then we can further define the candidate rate function
	\begin{equation}\label{rate function}
		I_{u_0} (u)=\inf_{\{g\in L^2_t H_x^r:\, u=\mathcal{G}^0(u_0, \text{Int}(g))\}} \bigg\{ \frac{1}{2} \int_{0}^{T} \|g(s) \|_{H^r}^2 \, ds \bigg\}.
	\end{equation}
	Here we assume that $\inf \emptyset=+\infty$. In the sequel, $\mathcal{E}_0^L$ and $\mathcal{E}^{T,L}$ defined in \eqref{ETL space} will stand for $\mathcal{E}_0$ and $\mathcal{E}$ in Hypothesis \ref{LDP hypothesis}, respectively.
	As $\mathcal{G}^0$ is a continuous map due to Lemma \ref{LDP convergence lemma}, we claim
	\begin{lemma}\label{hypothesis 1}
		For any $M<\infty$ and compact set $\mathcal{K} \subset \mathcal{E}_0^L$, the following set is a compact subset of $\mathcal{E}^{T,L}$:
		$$\Gamma_{\mathcal{K},M}:=\big\{\mathcal{G}^0(u_0, \text{Int}(g)):u_0\in \mathcal{K}, \, g\in  S^M(H^r)\big\}.$$
	\end{lemma}
	
	Next, we turn to verify the second condition of Hypothesis \ref{LDP hypothesis}. For $n \geq 1$, consider the following sequence of stochastic equations with initial value $\{u^n_0\}_n \subset L^2$ and $\{g^n\}_n \subset L^2_t H^r_x$:
	\begin{equation}\nonumber
		\begin{split}
			du^n+F_N(\|u^n\|_{H^1}) \Pi(u^n \cdot \nabla u^n) \, dt=-(-\Delta)^\Lambda u^n \, dt+  \Pi \Big[\circ d\Pi_n\big(\sqrt{\varepsilon_n}  \, W^r+\text{Int}(g^n) \big) \cdot \nabla u^n\Big],
		\end{split}
	\end{equation}
	where $\Lambda \in (1,2)$, $W^r$ is the space-time noise defined above \eqref{eq-LDP}, $d\Pi_n (\text{Int}(g^n))=\Pi_n g^n \, dt$ and $\varepsilon_n$ is the constant stated in \eqref{thetaR expression}, which vanishes  as  $n\rightarrow \infty$. Equation \eqref{2nd skeleton eq} can be written in the equivalent It\^o form:
	\begin{equation}\label{2nd skeleton eq}
		\begin{split}
			du^n+F_N(\|u^n\|_{H^1}) \Pi(u^n \cdot \nabla u^n) \, dt=&-(-\Delta)^\Lambda u^n \, dt+S_n(u^n) \, dt\\
			&+ \Pi \Big[ d\Pi_n\big(\sqrt{\varepsilon_n}  \, W^r+\text{Int}(g^n) \big) \cdot \nabla u^n\Big],
		\end{split}
	\end{equation}
	where  the It\^o-Stratonovich corrector has the following explicit form:
	$$S_n(u^n)=\frac{3\nu}{2}\varepsilon_n \sum_{|k|\leq n,i} |k|^{-2r}\Pi \big[\sigma_{k,i}\cdot \nabla \Pi(\sigma_{-k,i} \cdot \nabla u^n)\big].$$
	
	Thanks to the existence of $\Pi_n$, there are only finitely many Brownian motions. Based on the arguments in the proofs of Theorem \ref{well-posedness thm} and Lemma \ref{well posedness for skeleton eq}, we can use similar methods to prove that, for any fixed $n \geq 1$, equation \eqref{2nd skeleton eq} has a pathwise unique weak solution $u^n$. Noting that we have defined solution map $\mathcal{G}^n$ associated to equation \eqref{eq-LDP} above Theorem \ref{LDP thm}, the solution $u^n$ to \eqref{2nd skeleton eq} can be expressed in a similar way:
	$$u^n=\mathcal{G}^n\big(u_0^n, \sqrt{\varepsilon_n} \, W^r+\text{Int}(g^n)\big).$$
	
	Recall the spaces we defined in \eqref{ETL space} and Section \ref{subsec-weak convergence method}. The following lemma implies that the second condition of Hypothesis \ref{LDP hypothesis} holds.
	\begin{lemma}\label{hypothesis 2}
		Let $\{g^n\}_n \subset \mathcal{P}_2^M(H^r)$ and $\{u_0^n\}_n \subset \mathcal{E}_0^L$. Assume that, as $n\rightarrow \infty$, $u_0^n \rightharpoonup u_0$ in $L^2$ and $g^n$ converge in law, as $S^M(H^r)$-valued random elements, to $g$. Then the convergence
		$$\mathcal{G}^n\big(u_0^n, \sqrt{\varepsilon_n} \, W^r+\text{Int}(g^n)\big) \rightarrow \mathcal{G}^0\big(u_0, \text{Int}(g)\big)$$
		holds in law, in the topology of $\mathcal{E}^{T,L}$.
	\end{lemma}
	
	To prove this lemma, we first give an estimate for the martingale term
	$$M_t^n=\sqrt{\varepsilon_n} \int_{0}^{t} \Pi \big(d(\Pi_n W^r) \cdot \nabla u^n_s\big)=\sqrt{\frac{3\nu\varepsilon_n}{2} } \sum_{|k|\leq n,i} |k|^{-r} \int_{0}^{t} \Pi\big(\sigma_{k,i} \cdot \nabla u^n_s\big) \, dW^{k,i}_s.$$	
	
	\begin{lemma}\label{martingale estimate}
		For any $T>0$, $p\in [1,\infty)$, $\rho>\frac{5}{2}$ and $\eta \in (0, \frac{1}{2})$, we can find a finite constant $C>0$ depending on $T,\nu,\rho, \eta, p$ but independent of $n\geq 1$, such that
		$$\E \|M^n\|_{C^\eta H^{-\rho}}^{2p} \leq C \varepsilon_n^p L^{2p}.$$
	\end{lemma}
	
	\begin{proof}
		The proof is similar to that of \cite[Lemma 2.14]{GL24 LDP}, but we provide it for completeness.
		Denote $[M^n]_\rho$ as the quadratic variation associated to the norm $\|\cdot\|_{H^{-\rho}}$, then by the definition of $M^n$ and the divergence-free property of $\sigma_{k,i}$, it holds
		$$\frac{d}{dt}[M^n]_\rho=\frac{3\nu\varepsilon_n}{2} \sum_{|k|\leq n,i} |k|^{-2r} \big\|\Pi(\sigma_{k,i} \cdot \nabla u^n_t) \big\|_{H^{-\rho}}^2 \lesssim_\nu \varepsilon_n \sum_{k,i} \|\sigma_{k,i} \, u^n_t\|_{H^{1-\rho}}^2.$$
		Furthermore, we have
		$$\|\sigma_{k,i} \,  u^n_t\|_{H^{1-\rho}}^2 \lesssim \|e_k u^n_t\|_{H^{1-\rho}}^2=\sum_l |l|^{2(1-\rho)} \big|\<u^n_t, e_{l-k}\>\big|^2,$$
		and therefore for $\rho>\frac{5}{2}$, it holds
		$$\frac{d}{dt}[M^n]_\rho\lesssim_\nu \varepsilon_n \sum_{k,l} |l|^{2(1-\rho)} \big|\<u^n_t, e_{l-k}\>\big|^2 =\varepsilon_n \|u_t^n\|_{L^2}^2  \sum_{l} |l|^{2(1-\rho)} \lesssim_\rho \varepsilon_n L^2.$$
		By Burkholder-Davis-Gundy inequality, we can further obtain
		$$\E \big\|M_t^n-M^n_s \big\|_{H^{-\rho}}^{2p} \lesssim \E \Big(\big[M^n_\cdot-M^n_s\big]_\rho^p(t)\Big) \lesssim_{\nu,\rho} |t-s|^p \varepsilon_n^p L^{2p}.$$
		Applying Kolmogorov continuity theorem, we obtain the desired result.
	\end{proof}
	
	Now we can provide

\begin{proof}[Proof of Lemma \ref{hypothesis 2}]
As the proof is a little long, we divide it into three steps.
	
	\textbf{Step 1.} We first show that the laws of $\{u^n\}_n$, the weak solutions to \eqref{2nd skeleton eq}, are tight on $C^0_t H^{-\lambda}_x$ for any $\lambda>0$, which is the topology on $\mathcal{E}^{T,L}$. By \cite[p.90, Corollary 9]{JSimon}, we need to show $\P$-a.s., $ \{u^n\}_n \subset C^{\frac{1}{2}}_t H^{-\rho}_x \cap C^0_t L^2_x$ for a large $\rho>0$.	The second condition follows from our assumption on the initial value and classical energy estimate. So we only verify that $\{u^n\}_n \subset C^{\frac{1}{2}}_t H^{-\rho}_x$, $\P$-a.s., which is a direct corollary of
	\begin{equation}\label{holder norm}
		\E \, \|u^n_t-u^n_s\|_{H^{-\rho}} \leq C' |t-s|^{\frac{1}{2}}, \quad 0\leq s<t \leq T, \quad \forall n \in \mathbb{N}.
	\end{equation}
	The estimates for the terms concerning $g$ and cut-off are similar to those in Lemma \ref{LDP convergence lemma}, where we proved that the weak solution to equation \eqref{skeleton eq} belongs to $C^{\frac{1}{2}}_t H^{-\Lambda}_x$. As for the term $S_n(u^n)$ which is not involved therein,  \eqref{eq-S duality} implies that  $\|S_n(u^n)\|_{H^{-1}} \lesssim_\nu  \| u^n\|_{H^1}$. Hence it holds
	$$\E \int_{s}^{t} \big\|S_n(u^n)\big\|_{H^{-\rho}} \, dz \lesssim_\nu \, \E \int_{s}^{t} \|u^n\|_{H^1} \, dz \leq_\nu \, \|u_0\|_{L^2} |t-s|^{\frac{1}{2}}.$$
	The estimates for the martingale term have been given in Lemma \ref{martingale estimate}. Summarizing these discussions, we claim \eqref{holder norm} holds and therefore we can apply Aubin-Lions lemma to obtain that the laws of $\{u^n\}_n$ are tight on $C^0_t H_x^{-\lambda}$ for any $\lambda>0$.
	
	\textbf{Step 2.} Since we have verified the tightness of the laws of $\{u^n\}_n$, we can prove the existence of limit point. Applying Prohorov theorem and Skorohod representation theorem, we can find a new probability space $(\tilde{\Omega},\tilde{\mathcal{F}},\tilde{\P})$, on which there exist a family of independent Brownian motions $\tilde{W}^{n_j}=\{\tilde{W}^{n_j,k,i}: k\in \Z_0^3, i=1,2\}$ and a sequence of stochastic process $\{\tilde{u}^{n_j}\}_j$, together with limit $\tilde{u}$ and $\tilde{W}$, such that $\tilde{\P}$-a.s., $(\tilde{u}^{n_j}, \tilde{W}^{n_j})$ converges to $(\tilde{u},\tilde{W})$ in  $C^0_t H^{-\lambda}_x \times \mathcal{Y}$, where $\mathcal{Y}$ is defined as in the proof of Theorem \ref{well-posedness thm}. In the following proof, we will omit subscript $j$ and notation ``tilde" for the simplicity of notation.
	
	\textbf{Step 3.} Then we verify that $u$ is a weak solution to skeleton equation \eqref{skeleton eq}. Once it holds, by the uniqueness result in Lemma \ref{well posedness for skeleton eq}, we can assert that $u=\mathcal{G}^0\big(u_0, \text{Int}(g)\big)$.
	
	By the assumption that $\{u_0^n\}_n \subset \mathcal{E}_0^L$ and $u_0^n \rightharpoonup u_0$ in $L^2$, it holds $u^n_0 \rightarrow u_0$ in $H^{-\lambda}$.	For the martingale part, by Lemma \ref{martingale estimate}, it will vanish as $n\rightarrow \infty$ since $\varepsilon_n \sim n^{2r-3}$ for $r\in (0,\frac{3}{2})$. And the convergence for the term $-(-\Delta)^\Lambda u^n$ is trivial since it is linear. As for the nonlinear term, similar discussions can be found in the proof of Lemma \ref{well posedness for skeleton eq}. So we only discuss the remaining two terms here.
	
	For the term corresponding to $g$, we aim to prove that for any divergence-free $\phi\in C^\infty(\T^3,\R^3)$,
	$$\int_{0}^{t} \big\<\Pi\big(\Pi_n g^n \cdot \nabla u^n\big),\phi\big\> \, ds \rightarrow \int_{0}^{t} \big\< \Pi(g\cdot \nabla u), \phi\big\> \, ds.$$
	In fact, by the divergence-free properties of $g^n$, $g$ and $\phi$, we only need to consider
	$$\aligned
	&\quad \ \bigg|\int_{0}^{t} \big\<u^n\Pi_n g^n-ug, \nabla \phi\big\> \, ds\bigg|  \\
	&\leq \bigg|\int_{0}^{t} \big\< \Pi_n g^n (u^n-u), \nabla\phi \big\> \, ds\bigg|+\bigg|\int_{0}^{t} \big\<u \big(\Pi_n g^n- \Pi_n g\big), \nabla\phi \big\> \, ds\bigg|+\bigg|\int_{0}^{t} \big\<u(\Pi_n g-g), \nabla\phi \big\> \, ds \bigg|.
	\endaligned$$
	For $g^n\in \mathcal{P}_2^M(H^r)$, we use Cauchy-Schwarz inequality to get
	$$\aligned
	\bigg|\int_{0}^{t} \big\< \Pi_n g^n (u^n-u), \nabla\phi \big\> \, ds\bigg| &\leq \int_{0}^{t} \|u^n-u\|_{H^{-\frac{r}{2}}} \|\Pi_n g^n \cdot \nabla\phi  \|_{H^\frac{r}{2}} \, ds \\
	&\lesssim \int_{0}^{t} \|u^n-u\|_{H^{-\frac{r}{2}}}  \|\Pi_n g^n\|_{H^r} \|\nabla\phi\|_{H^\frac{3-r}{2}} \, ds\\
	&\leq M^\frac{1}{2} T^\frac{1}{2} \|u^n-u\|_{C^0_t H^{-\frac{r}{2}}_x} \|\phi\|_{H^\frac{5-r}{2}}.
	\endaligned $$
	Since we have proved in Step 2 that $u^n$ converges to $u$ in $C^0_t H^{-\lambda}_x$ for any $\lambda>0$, this term vanishes as $n \rightarrow \infty$. To prove the second term also tends to zero as $n \rightarrow \infty$, by the assumption that $g^n \rightarrow g$ in the weak topology of $L_t^2 H^r_x$, it is enough to verify that $u \cdot \nabla \phi \in L^2_t H^{-r}_x$, which can be deduced from the following estimate:
	$$\|u\cdot\nabla\phi\|_{L^2_t H^{-r}_x} \lesssim \Big(\int_{0}^{t} \|u\|^2_{L^2} \|\nabla\phi\|^2_{H^{\frac{3}{2}-r}} \, ds\Big)^\frac{1}{2} \leq T^\frac{1}{2} \|u_0\|_{L^2} \|\phi\|_{H^{\frac{5}{2}-r}}<+\infty.$$
	For the last term, by H\"older's inequality and Sobolev embedding theorem, it holds
	$$\bigg|\int_{0}^{t} \big\<u(\Pi_n g-g), \nabla\phi \big\> \, ds \bigg| \leq \int_{0}^{t} \|u\|_{L^2} \|\Pi_n g-g\|_{L^4} \|\nabla\phi\|_{L^4} \, ds
	\leq \|\phi\|_{H^\frac{7}{4}}  \|u_0\|_{L^2} \int_{0}^{t} \|\Pi_n g-g\|_{H^{\frac{3}{4}}}  \, ds,$$
	which vanishes as $n \rightarrow \infty$. Summarizing the above arguments, we proved the convergence for the term concerning $g$.
	
	For the term $S_n(u^n)$, we will decompose the difference as two parts:
	$$\Big\|S_n(u^n)-\frac{3\nu}{5} \Delta u \Big\|_{H^{-\lambda}} \leq  \Big\|S_n(u^n)-\frac{3\nu}{5} \Delta u^n \Big\|_{H^{-\lambda}}+\frac{3\nu}{5} \big\|\Delta (u^n-u)\big\|_{H^{-\lambda}}.$$
	The convergence for the second term is simple as $u^n$ converges to $u$ in $C^0_t H_x^{-\lambda}$ for any $\lambda>2$. For the first term, using Theorem \ref{S limit thm} with $b=0$ and $\alpha=\lambda-2 \in [0,1]$, we have
	$$\Big\| S_n(u^n)-\frac{3\nu}{5} \Delta u^n \Big\|_{H^{-\lambda}} \lesssim \nu n^{2-\lambda} \|u^n\|_{L^2} \lesssim \nu n^{2-\lambda} \|u_0\|_{L^2},$$
	which tends to 0 as $n\rightarrow \infty$ for $\lambda\in (2,3]$. Summarizing the above arguments, we prove that $u=\mathcal{G}^0\big(u_0, \text{Int}(g)\big)$ and therefore get the desired conclusion of Lemma \ref{hypothesis 2}.
\end{proof}
	
Finally, combining this lemma with Lemma \ref{hypothesis 1}, the proof of Theorem \ref{LDP thm} is complete.

	\bigskip
	\noindent\textbf{Acknowledgements.} The second author is support by the National Key R\&D Program of China (Nos. 2020YFA0712700, 2024YFA1012301), and the National Natural Science Foundation of China (Nos. 12090010, 12090014).

\end{document}